\documentclass[11pt]{article}
\usepackage{amsfonts}
\usepackage{amsmath}
\usepackage{amsthm}
\usepackage{amssymb}
\usepackage{graphicx}
\usepackage{sidecap}
\usepackage{caption}
\usepackage{array}
\usepackage{subcaption}
\usepackage{empheq}
\usepackage{xcolor}
\usepackage{float}
\usepackage{dsfont}
\usepackage{hhline}
\usepackage{blkarray}
\usepackage{hyperref}
\usepackage[square,numbers]{natbib}
\usepackage{bm}
\usepackage{multirow}
\usepackage{multicol}
\usepackage[margin=1.00in]{geometry}
\graphicspath{{images/}}
\allowdisplaybreaks

\newtheorem{corollary}{Corollary}[section]

\usepackage{hhline}

\usepackage{appendix}

\usepackage{algorithm}
\usepackage{algpseudocode}
\algrenewcommand\algorithmicrequire{\textbf{Input:}}
\algrenewcommand\algorithmicensure{\textbf{Output:}}

\usepackage{chngcntr}
\counterwithin{figure}{section}
\counterwithin{table}{section}
\counterwithin{equation}{section}

\usepackage{tikz}
\usetikzlibrary{shapes.geometric,arrows,automata}
\usetikzlibrary{calc}
\usetikzlibrary{decorations.pathreplacing} 

\tikzstyle{box} = [rectangle, rounded corners, minimum width=3cm, minimum height=1cm,text centered,text width=3cm, draw=black]
\tikzstyle{longbox} = [rectangle, rounded corners, minimum width=6cm, minimum height=1cm,text centered,text width=6cm, draw=black]

\tikzstyle{end} = [rectangle, rounded corners, minimum width=5cm, minimum height=1cm,text centered,text width=5cm, draw=black,fill=red!30]
\tikzstyle{arrow} = [thick,->,>=stealth]

\tikzstyle{line} = [draw, -latex]

\tikzstyle{startstop} = [rectangle, rounded corners, 
minimum width=3cm, 
minimum height=1cm,
text centered, 
draw=black]

\tikzstyle{start} = [rectangle, rounded corners, 
minimum width=3cm, 
minimum height=1cm,
text centered, text width=3.5cm,minimum width=3.5cm,
draw=black]
\tikzstyle{io} = [trapezium, 
trapezium stretches=true, 
trapezium left angle=70, 
trapezium right angle=110, 
minimum width=3.5cm, 
minimum height=1cm, text centered, 
draw=black]

\tikzstyle{process} = [rectangle, 
minimum width=3cm, 
minimum height=1cm, 
text centered, 
text width=3cm, 
draw=black]

\tikzstyle{smallprocess} = [rectangle, 
minimum width=1.5cm, 
minimum height=1cm, 
text centered, 
text width=1.5cm, 
draw=black]

\tikzstyle{decision} = [rectangle, 
minimum width=3cm, 
minimum height=1cm, 
text centered,
draw=black]
\tikzstyle{arrow} = [thick,->,>=stealth]



\newtheorem{theorem}{Theorem}[section]

\newtheorem{remark}{Remark}

\newcommand{\bit}{\begin{itemize}}
\newcommand{\eit}{\end{itemize}}
\newcommand{\ben}{\begin{enumerate}}
\newcommand{\een}{\end{enumerate}}
\newcommand{\beq}{\begin{equation}}
\newcommand{\eeq}{\end{equation}}
\newcommand{\beqa}{\begin{eqnarray*}}
\newcommand{\eeqa}{\end{eqnarray*}}
\newcommand{\bc}{\begin{center}}
\newcommand{\ec}{\end{center}}

\DeclareMathOperator*{\argmin}{argmin}
\DeclareMathOperator*{\argmax}{argmax}
\newcommand{\bx}{\textbf{x}}
\newcommand{\bX}{\textbf{X}}
\newcommand{\bxi}{\bm{\xi}}

\makeatletter
\def\hlinewd#1{%
\noalign{\ifnum0=`}\fi\hrule \@height #1 \futurelet
\reserved@a\@xhline}
\makeatother

\usepackage{cellspace}
\title{Towards Robust Data-Driven Automated Recovery of Symbolic Conservation Laws from Limited Data}
\author{Tracey Oellerich$^{(1)}$, Maria Emelianenko$^{(1)}$\\ $^{(1)}$Department of Mathematical Sciences, George Mason University}

\begin{document}

\maketitle
\section*{Abstract}
Conservation laws are an inherent feature in many systems modeling real world phenomena, in particular, those modeling biological and chemical systems. If the form of the underlying dynamical system is known, linear algebra and algebraic geometry methods can be used to identify the conservation laws. Our work focuses on using data-driven methods to identify the conservation law(s) in the absence of the knowledge of system dynamics. Building in part upon the ideas proposed in \cite{Kaiser18}, we develop a robust data-driven computational framework that automates the process of identifying the number and type of the conservation law(s) while keeping the amount of required data to a minimum. We demonstrate that due to relative stability of singular vectors to noise we are able to reconstruct correct conservation laws without the need for excessive parameter tuning. While we focus primarily on biological examples,  the framework proposed herein is suitable for a variety of data science applications and can be coupled with other machine learning approaches. 

\section{Introduction}

Due to the importance of identifying conserved quantities in a variety of practical applications, for decades this area has enjoyed much attention in applied mathematical literature. Analytically, linear conservation laws can be discovered from a given dynamical system by solving a linear system of equations formed using a stoichiometric matrix \cite{Feinberg, Oellerich20, Dickenstein16}. Nonlinear relations can be found using more sophisticated methodology as for instance discussed in \cite{Bressan11}. Methods in Lie group theory have also been heavily utilized for determining conservation laws \cite{Popovych20, Anderson96, Peng17,Khamitova09,Olver93}.  

Approaching this problem computationally, we see there are multiple algorithms designed to learn conservation laws given a dynamical system input. In \cite{Holiday19}, manifold learning is applied to both learn conservation laws and identify reduced state variables. AI-Poincar\'{e} 2.0\cite{Liu22} takes differential equation inputs and uses trained neural nets and manifold learning to identify the number of conservation laws as well as their symbolic representation. Eliminating the need for a black-box methodology, the Sparse Invariant Detector (SID)\cite{Liu23} uses linear algorithms to discover conservation laws. However, in the age of data-driven modeling, we want to tackle a more challenging question of recovering these relations without knowing the system dynamics a priori.

Specialized approaches have been developed to specifically recover underlying Hamiltonian functions\cite{Greydanus19, Toth19}.  More general approaches, such as AI-Poincar\'{e} \cite{Liu21} and ConservNet \cite{Ha21}, use neural networks to recover conservation laws from observation data. In \cite{Mototake21}, a deep neural net structure is paired with Noether's theorem to extract physical information about the underlying physical system. Manifold learning with optimal transport \cite{Lu23} provides a geometric approach shown to be robust to noise. In \cite{Arora23}, machine learning is used to both identify the conservation laws via first integrals as well as the underlying dynamical system structure. Koopman theory coupled with data-driven regression provides a symbolic representation of the conserved quantities \cite{Kaiser18}. Typically, these approaches can require anywhere from 100 points \cite{Arora23} to upwards of $10^3$ points\cite{Liu22}. Higher-order systems may require additional points for accurate recovery, although it is unclear what the relationship is amongst order and necessary observations. Ability to handle noisy experimental data also remains a outstanding challenge that needs to be addressed.

The SVD-based null space estimation method that we adopt in this paper is relatively robust and less demanding in terms of data and underlying structure requirements. In fact, as shown herein, in many cases only 20 points is sufficient for accurate recovery of the underlying conserved structures even in a high-noise environment. In addition to keeping the observable data size requirements to a minimum, we also aim at learning the functional (symbolic) form of the conservation law(s) which typically is not a direct outcome for the algorithms mentioned above, with the notable exception of \cite{Kaiser18} which utilizes the same sparse identification paradigm as used in this work. Some well documented challenges associated with this approach include the need to choose appropriate functional libraries and corresponding optimality thresholds. 

We develop a unified data-driven approach to automatically select the optimal learning library and make an informed choice regarding the best polynomial degree to be used. The main idea is to perform a sequence of data transformations including SVD decomposition followed by echelon form reduction and let numerical data drive the identification of the singular value gap and inform consecutive selection of the optimization parameters. We demonstrate the power of our method on several classes of problems including automatic detection of multiple conservation laws and nonlinear conservation with applications ranging from physics to biology. 

To summarize, we propose a general methodology aimed at learning symbolic forms of conservation law(s) that allows to: (1) limit the amount of data needed for robust discovery of conservations in the presence of noise and without the knowledge of the underlying dynamics, and (2) automate the process of selecting optimal libraries and optimization thresholds based solely on the estimated noise level and the singular values of the library matrix associated with observation data.

We note that due to the relative ease of the SVD and LU-based implementation, the proposed framework lends itself well to being integrated with other learning approaches. For example, once the conservation laws have been identified from the data, they can then be used to aid in dynamical system recovery. Machine learning methods such as \cite{Zhang23, Readshaw23, Wu21, Lee21} impose conservation laws as constraints in the recovery process. On a similar note, data-driven methods such as constrained sparse Galerkin \cite{Loiseau18}, DMDc \cite{Proctor16}, SINDy-c \cite{Brunton16b}, and generalized Koopman Theory \cite{Proctor18} allow for conservation laws to be used as constraints in their optimization algorithms. 

This work is organized as follows. In Section \ref{sec:notation}, we consider systems with unknown structure which may or may not be conservative.  Using predefined candidate libraries provided by the user, we develop a general SVD-based to infer conservation laws or show none exist for the given input. As this approach is dependent on the robustness of the SVD to noise, in Section \ref{sec:noise} we derive error bounds for a class of library functions and perform numerical validation experiments. To alleviate the burden of specific library choice on the user, in Section \ref{sec:identify} we propose an algorithm allowing users to input multiple options and compare the results to make an informed decision. Using benchmark examples defined in Section \ref{sec:examples}, in Section \ref{sec:testing} we apply the algorithm in both low data noisy environment. Concluding remarks can then be found in Section \ref{sec:discussion}.

\section{Mathematical Notation and Methodology}
\label{sec:notation}
For this work, we will consider systems modeled using differential equations. Let $\textbf{x}(t) = \textbf{x} = \begin{bmatrix} x_1, x_2, x_3,..., x_n\end{bmatrix} \in \mathbb{R}^n$ be functions of time for each state variable $x_i = x_i(t)$. Consider the differential equation system shown in Equation \eqref{eq:basesystem}. Each state variable can therefore be modeled by the differential equation, $\dot{x}_i(t) = f_i(\textbf{x}(t))$. 

\begin{equation}
\begin{array}{c}
    \dot{\textbf{x}}(t) = \textbf{f}(\textbf{x}) \\
\end{array}
\label{eq:basesystem}
\end{equation}

The goal of this paper is to formulate a robust data-driven method for finding all possible conservation laws and their respective functional forms for the system\eqref{eq:basesystem} under consideration. We will seek conservation laws in the form  $g_k(\textbf{x})=C_k, k=1,\ldots K$ and $g_k:\mathbb{R}^n\to \mathbb{R}$. To simplify notation, we will use $g(\textbf{x})$ to represent a single conservation law, using subscripts as necessary. 

The data-driven nature of this approach can be articulated as follows. Following the formalism introduced in \cite{Kaiser18} we will seek an approximation for $g(\textbf{x})$ in the form specified in \eqref{eq:linapprox}. 

\begin{equation}
    g(\textbf{x)} = C \approx \Theta(\textbf{x})\bxi
    \label{eq:linapprox}
\end{equation}
where $ \Theta(\textbf{x}) = [\bm{\theta}_1(\textbf{x}), \bm{\theta}_2(\textbf{x}) \ldots,  \bm{\theta}_p(\textbf{x}) ]^T\in \mathbb{R}^{1\times p}$ and $\bm{\xi} = [\xi_1, \ldots \xi_p ]^T\in \mathbb{R}^{p}$.

In this setup, $\Theta(\textbf{x})$ is a symbolic library of candidate functions, $\{\bm{\theta}_i(\bx)\}_{i,\ldots p}$ which later on will be evaluated at $N$ data points to construct the data matrix. The choice of the library, in general left to the user, is motivated by the functional forms one expects to see in a particular application. The corresponding $\bm{\xi}\in \mathbb{R}^{p}$ is the vector of weights associated with the library terms that we are trying to find. The $\Theta$-library can be constructed in multiple ways as it is truly up to the user's discretion which terms to include. In this work, we will advocate for a robust library selection approach to alleviate the burden of placing this decision on the user. 

\begin{remark}In our considerations, the term $\mathbf{{x}}^k$ denotes the multi-index corresponding to all possible order $k$ monomials of  $\mathbf{x}$. For example, if $\mathbf{x} = [x_1, x_2]$, then $\textbf{x}^3 = \begin{bmatrix} x_1^3, & x_1^2x_2, &x_1x_2^2, &x_2^3\end{bmatrix}$.
\end{remark}

An example of a candidate $\Theta$-library for a system with two state variables, $\textbf{x} = [x_1, x_2]$ is given by 
\begin{align*}
\bm{\Theta}(\textbf{x}(t)) &=\begin{bmatrix}  \textbf{x} & \textbf{x}^2 & \sin(\textbf{x}) & \ln(\textbf{x}) \end{bmatrix}
=\begin{bmatrix}  x_1 & x_2 & x_1^2 & x_1x_2 & x_2^2 & \sin(x_1) & \sin(x_2) & \ln(x_1) & \ln(x_2) \end{bmatrix}
\end{align*}
with $p=9$. 

\begin{remark}
Natural domain considerations must be applied when considering the choice of the library functions. For instance, logarithmic functions should only be included in the library if the data is strictly positive. For example, we primarily consider biological networks in this work which should have strictly positive values.
\end{remark}
\begin{remark}
 In contrast with SINDy-based work \cite{Brunton16}, in our study we will omit the constant $\textbf{1}$ term that is normally present and include the $\ln$ terms which have been shown to appear in nonlinear conservation laws associated with chemical reaction systems \cite{Emelianenko14}. We will comment on the first note after the construction of the algorithm.
 \end{remark}

By symbolically taking the derivative of \eqref{eq:linapprox} with respect to time we get

\begin{equation}
    0=\dfrac{d}{dt}C=\dfrac{d}{dt}g(\bx) = 
    \left(\frac{d \Theta}{dt}\right)\bxi  =
    \Gamma(\bx, \dot{\bx})\bxi 
    \label{eq:formal}
\end{equation}
where  $\Gamma(\bx, \dot{\bx}) = [\dot \theta_1(\bx)\ldots \dot \theta_p(\bx)] \in \mathbb{R}^{1\times p}$ where $\theta_i$ corresponds to the i-th $\Theta$-library function. For the sample $\Theta$-library library given, the associated $\Gamma$-library is
\begin{equation}
\Gamma(\textbf{x},\dot{\textbf{x}})
=\begin{bmatrix}  
\dot{x}_1 & \dot{x}_2 & 2x_1\dot{x}_1 & \dot{x}_1x_2 +x_1\dot{x}_2& 2x_2\dot{x}_2 & \dot{x}_1\cos(x_1) & \dot{x_2}\cos(x_2) & \dfrac{\dot{x}_1}{x_1} & \dfrac{\dot{x}_2}{x_2} 
\end{bmatrix}.
\label{eq:exGamma}
\end{equation}

Let $\textbf{X} = [\textbf{x}(t_1), \textbf{x}(t_2), ..., \textbf{x}(t_N)]^T \in \mathbb{R}^{N\times n}$ be time series data for the system with time derivative data given by $\dot{\textbf{X}} = [\dot{\textbf{x}}(t_1), \dot{\textbf{x}}(t_2), \ldots, \dot{\textbf{x}}(t_N)]$. If the derivative data is not available, then $\dot{\textbf{X}}$ can be numerically approximated from $\textbf{X}(t)$ \cite{Iserles08,Nocedal06}. For systems with an abundance of noise, recent techniques in numerical differentiation can be used \cite{Chartrand11,vanBreugel20}. 
 
Adding the time series data to Equation \eqref{eq:formal}, we have the over-determined least squares problem 
\begin{equation}
\Gamma(\bX,\dot{\bX)} \bxi = 0
\label{eq:findcons}
\end{equation}
where $\Gamma \in \mathbf{R}^{N\times p}$ is the data matrix corresponding to the symbolic derivative library $\Gamma(\bx,\dot \bx)$ defined above. 

We aim at finding a nontrivial minimal norm solution to the least squares problem 
\begin{equation}
\xi= \argmin_{\bxi'}{||\Gamma(\bX,\dot \bX)\bxi'||_2}
\label{eq:minimization}
\end{equation}

which is equivalent to solving the null space problem for the matrix $\Gamma$. This problem can be solved using singular value decomposition (SVD) of $\Gamma = USV^T$, where $S\in \mathbb{R}^{N\times p}$ contains decreasing singular values $\sigma_i\geq 0$ on the diagonal,  $U\in \mathbb{R}^{N\times N}$ and $V\in \mathbb{R}^{p\times p}$ are unitary matrices corresponding to the left and right singular vectors, respectively \cite{Demmel,Golub13}. Alternatively, there are more algorithmic approaches which can be employed to find the null space (e.g. \cite{Kaiser18, Qu16, Gottlieb10}), however here we explore an SVD approach as an alternative that may pave the way towards a robust user-agnostic strategy. We aim at minimizing complexity and hyperparameters, while designing a method able to work with relatively small datasets. 

Based on the SVD methodology, identifying the conservation law then translates into finding all right singular vectors corresponding to zero or close to zero singular values based on a certain cutoff. Figure \ref{fig:flowchart1} gives a schematic of the process we intend to use for finding conservation law(s) for given data and chosen $\Theta$-library.  If the process fails to produce a $\sigma_i\approx 0$, then there are three possibilities: (1) the starting library does not contain the appropriate terms, (2) inadequate data due to noise or amount, or (3) the system does not contain a conservation law. 

Main open questions in the outlined methodology that we aim to address below are as follows. First, how does one appropriately choose a $\Theta$-library. In some cases, for instance when dealing with biochemical reaction data, the simplest form of the conservation law, that is a only first order polynomials, could be the best fit. However, as we will see later, there are instances where more complex functions are needed to obtain conservation. The second question is identifying whether or not a system naturally contains multiple conservation laws and weeding out spurious solutions. Finally, we are interested in investigating the minimal amount of data necessary for performing this analysis given an accuracy guarantee. 

\begin{figure}[H]
\begin{center}









\includegraphics[width = 1.0\textwidth]{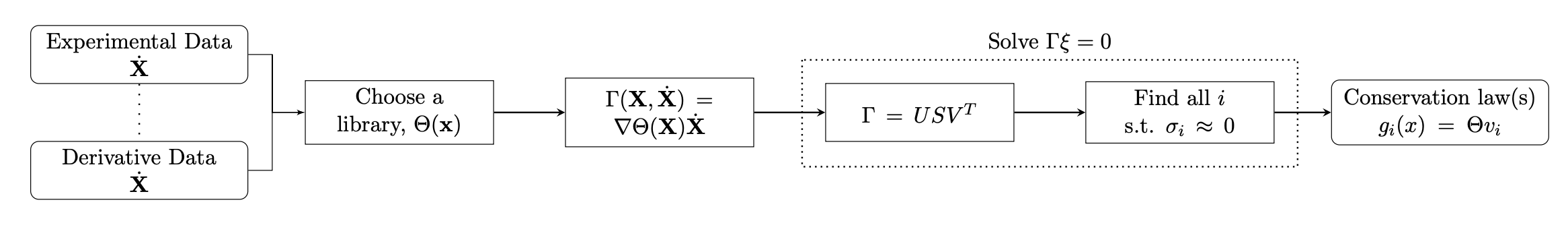}
\caption{Flowchart detailing the process of identifying conservation laws from data.} 
\label{fig:flowchart1}
\end{center}
\end{figure}

\section{Benchmark Examples}
\label{sec:examples}
In the following sections, we will be studying several networks which all contain conserved quantities. These benchmark examples can be found in Table \ref{tab:examples} and are referred to throughout the remainder of the work. Example 1, the Volpert Network, contains second order reactions, but only a single linear conservation law in all variables. Example 2 contains a mix of first and second order reactions and has two governing conservation laws, neither containing all state variables. Example 3 comes from chemistry and models the chemical oxidation process. The conservation law in this system contains a nonlinear logarithmic terms, making it harder to identify using traditional methods. Finally, we examine a degrading network in Example 4 devoid of conservation laws.  

\begingroup
\begin{table}[h]
    \centering
    \caption{Network examples used throughout this work. Data for these systems is generated using the provided coefficients.}
    
    \renewcommand{\arraystretch}{1.25}
    \scalebox{0.75}{
    \begin{tabular}{|m{2.02cm}|m{1.75cm}m{3.3cm}m{5cm}m{3cm}|}
    \hline 
    \textbf{Example} & \textbf{Variables} & \textbf{Reactions}& \textbf{Equations} & \textbf{Coefficients Used}\\
    \hline
    &&&& \\
    \textbf{1:}Volpert Network\cite{Mahdi17} & 
    $\begin{bmatrix} x_1\\x_2\\x_3\end{bmatrix}$ &
    {$\!
    \begin{aligned}
    &X_1+X_2\xrightarrow{k_1} 2X_2 \\    
    &X_2+X_3\xrightarrow{k_2} 2X_3 \\
    &X_3+X_1\xrightarrow{k_3} 2X_1
    \end{aligned}
    $} &
    {$\!
    \begin{aligned}
        &\dot{x}_1 = k_3x_1x_3-k_1x_1x_2&\\
        &\dot{x}_2 = k_1x_1x_2-k_2x_2x_3&\\
        &\dot{x}_1 = k_2x_2x_3-k_3x_1x_3&\\[.25em]
        &x_1+x_2+x_3 = C&
    \end{aligned}
    $}&    
    {$\!
    \begin{aligned}
    [&k_1,k_2,k_3] \\&= [1,1,1]\\
    \textbf{x}_0 &= [1,2,0.5]
    \end{aligned}
    $}
    \\ &&&& 
    \\ \hline
    &&&& \\
    \textbf{2:}Two \newline Conservation Laws& 
    $\begin{bmatrix} x_1\\x_2\\x_3\\ x_4\end{bmatrix}$ &
    {$\!
    \begin{aligned}
    &X_1+X_2\xrightarrow{k_1} X_3\\
    &X_3\xrightarrow{k_2} X_1+X_2 \\
    & X_3\xrightarrow{k_3} X_1+X_4
    \end{aligned}
    $}&
    {$\!
    \begin{aligned}
    \dot{x}_1 &= -k_1x_1x_2+k_2x_3+k_3x_3\\
    \dot{x}_2 &= -k_1x_1x_2+k_2x_3\\
    \dot{x}_3 &= k_1x_1x_2-k_2x_3-k_3x_3\\
    \dot{x}_4 &= k_3x_3\\[.25em]
    x_1&+x_3 = C_1\\
    x_2&+x_3+x_4 =C_2
\end{aligned}
    $} &
{$\!
    \begin{aligned}
    [&k_1,k_2,k_3] \\&= [1,1,1]\\
    \textbf{x}_0 &= [1,2,0.5,0.3]
    \end{aligned}
    $}
    \\ &&&& \\ \hline
&&&& \\
    \textbf{3:}Chemical Oxidation Network\cite{Emelianenko14} &
    $\begin{bmatrix} x\\y\\S\end{bmatrix}$& 
   {$\!
    \begin{aligned}
        &X+H_2O_2\xrightarrow{k_I} Y \\
        &Y+S\xrightarrow{k_{2}} X+P \\
        & Y\xrightarrow{k_3} \tilde{Y}\\
    \end{aligned}
    $}  &
       {$\!
    \begin{aligned}
        \dot{x} &= -k_1x+k_2yS\\
        \dot{y} &= k_1x-k_2yS-k_3y\\
        \dot{S} &= -k_2yS\\[.25em]
        x&+y-\dfrac{k_3}{k_2}\ln(S)= C
        \end{aligned}
    $} &
    {$\! 
    \begin{aligned}
    [&k_1,k_2,k_3] \\&= [2,0.4,1]\\
    \textbf{x}_0 &= [0.75,0.5,2]
    \end{aligned}
    $} 
    \\ &&&& \\ \hline
    &&&& \\
    \textbf{4:}No \newline Conservation Laws&
    $\begin{bmatrix} x_1\\x_2\end{bmatrix}$ &
    {$\!
    \begin{aligned}
&X_1\xrightarrow{k_1} 0 \\
& X_1+X_2 \xrightleftharpoons[k_3]{k_2} 2X_1
    \end{aligned}
    $}&
    {$\!
    \begin{aligned}
    &\dot{x}_1(t) = -k_1x_1+k_2x_1x_2\\
    &\hspace{4em}-k_3x_1^2\\
    &\dot{x}_2(t) =  -k_2x_1x_2+k_3x_1^2
    \label{eq:noncons}
    \end{aligned}
    $}&
{$\!
    \begin{aligned}
    [&k_1,k_2,k_3] \\&= [1,2,2]\\
    \textbf{x}_0 &= [1,2]
    \end{aligned}
    $}
    \\ 
    &&&& \\
    \hline
    \end{tabular}
    \label{tab:examples}
}
\end{table}
\endgroup

\section{Perturbation analysis}
\label{sec:noise}
In order to develop robust and stable numerical algorithm we need to understand the effect of noise on the solution produced by the SVD-based null-space estimation. There is abundance of literature devoted to the subject of studying perturbed rank-deficient systems (see e.g. \cite{Stewart06, Demmel} and references therein). Our goal in this work is to adapt existing theoretical results to the specific case at hand related to the form of the library $\Gamma$ and understand their implications for the sake of developing a robust numerical framework.    

The following classical result lays the foundation of our study. 
\begin{theorem}[Weyl \cite{Weyl, Lawson95}]
For any additive perturbation $\tilde A = A + 
\mathcal{E}$, the following bound on singular values $\sigma_i(A)$ is valid:
    \begin{equation}
\begin{array}{r@{}l}
&\varepsilon_{\sigma_i} = |\tilde{\sigma}_i - \sigma_i| \leq \lVert \mathcal{E} \rVert_2\hspace{2em} \forall i
\end{array}
\label{eq:sig_bounds}
\end{equation}
\end{theorem}

Weyl's theorem posits that the singular values of a matrix are perfectly conditioned in a sense that no singular value can move more than the norm of the additive perturbation. While powerful, this result is not practically useful. However, we can go further due to the special structure afforded by the $\Gamma$ matrix construction.

Let the exact data $\textbf{x}$ be corrupted by noise $\varepsilon_x$, say $\tilde{\textbf{x}} = \textbf{x}+\varepsilon_x$ with the corresponding derivative approximation,  $\tilde{\dot{\textbf{x}}} = \dot{\textbf{x}} +\varepsilon_{\dot x}$. Numerical differentiation is a complex problem in itself. Options include standard finite difference approximation (poor choice for sparse data), iterative finite differences, total variation regularization (TVR), Tikhonov regularization, and the Savitzky-Golay filter.  In our considerations, we will primarily be utilizing Tikhonov regularization and the Savitzky-Golay filter using the implementation provided in \cite{ Breugel22,Wagner20,rdiff}. Resulting additive perturbation of the library matrix is denoted as $\tilde{\Gamma} = \Gamma + \mathcal{E}_{\Gamma}$.  

First, we perform numerical experiments to understand the sensitivity of singular values of $\Gamma$ to i.i.d. Gaussian additive noise. We use synthetic data generated for each of the benchmark ODE systems using 200 points on the interval $[0,2]$, and choosing 100 interior points to alleviate boundary effects in derivative approximation which is calculated using Tikhonov regularization.

Figure \ref{fig:sing_error_all} shows the relationship between the variance of the noise and the corresponding numerically estimated error in singular values for each of the given benchmark examples.  It is clear that all errors in the singular values fall directly below $\mathcal{E}_{\Gamma}$, as expected by Equation \eqref{eq:sig_bounds}. However, these results do not directly provide intuition on how this bound is related to the observable quantities such as the noisy data error $\varepsilon_x$ or noisy derivative approximation error $\varepsilon_{\dot x}$. We also know that these dependencies will differ depending on the choice of the libraries $\Theta$ and the functional forms included therein. The following straightforward result provides this connection for a class of polynomial library functions.

\begin{figure}[!]
    \centering
    \normalsize{\textbf{Example 1: Volpert Network}}
    \includegraphics[width = 0.86\textwidth]{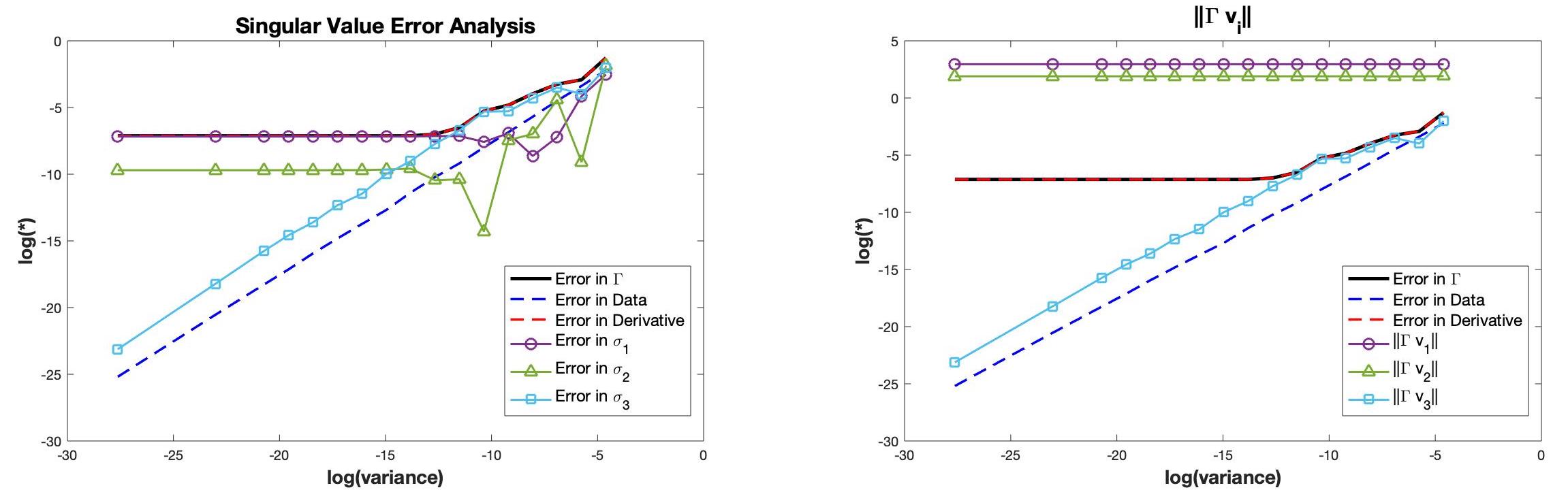}\\
    \normalsize{\textbf{Example 2: Network  with Two Conservation Laws }}
    \includegraphics[width = 0.86\textwidth]{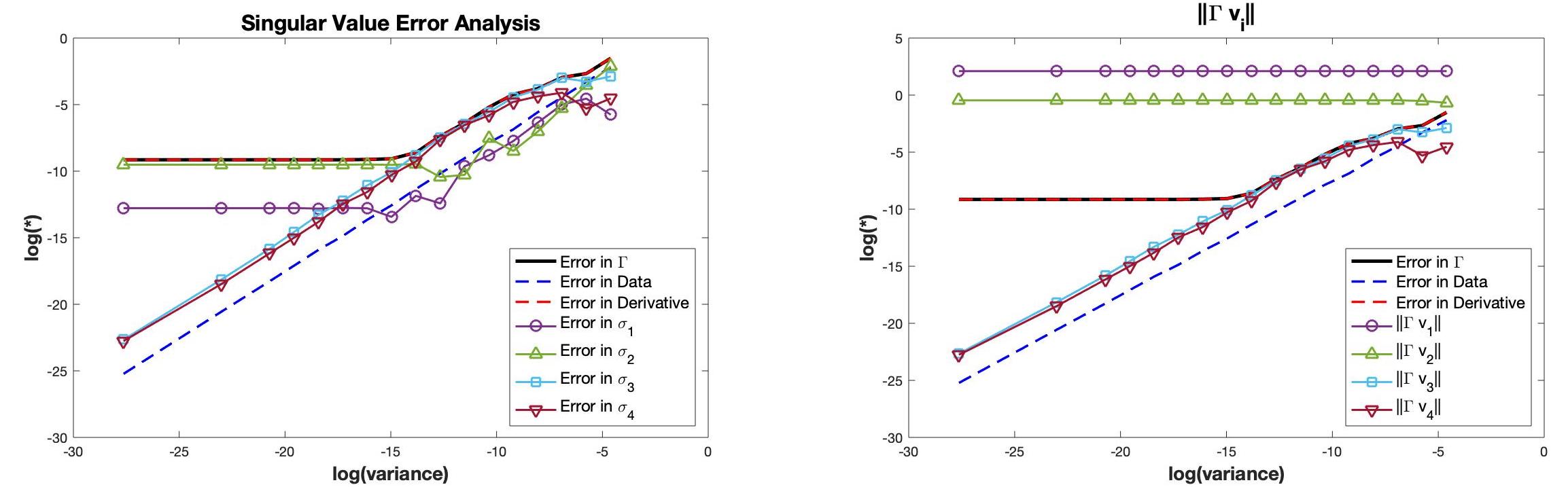}\\
   \normalsize{\textbf{Example 3: Chemical Oxidation Network}}
    \includegraphics[width = 0.86 \textwidth]{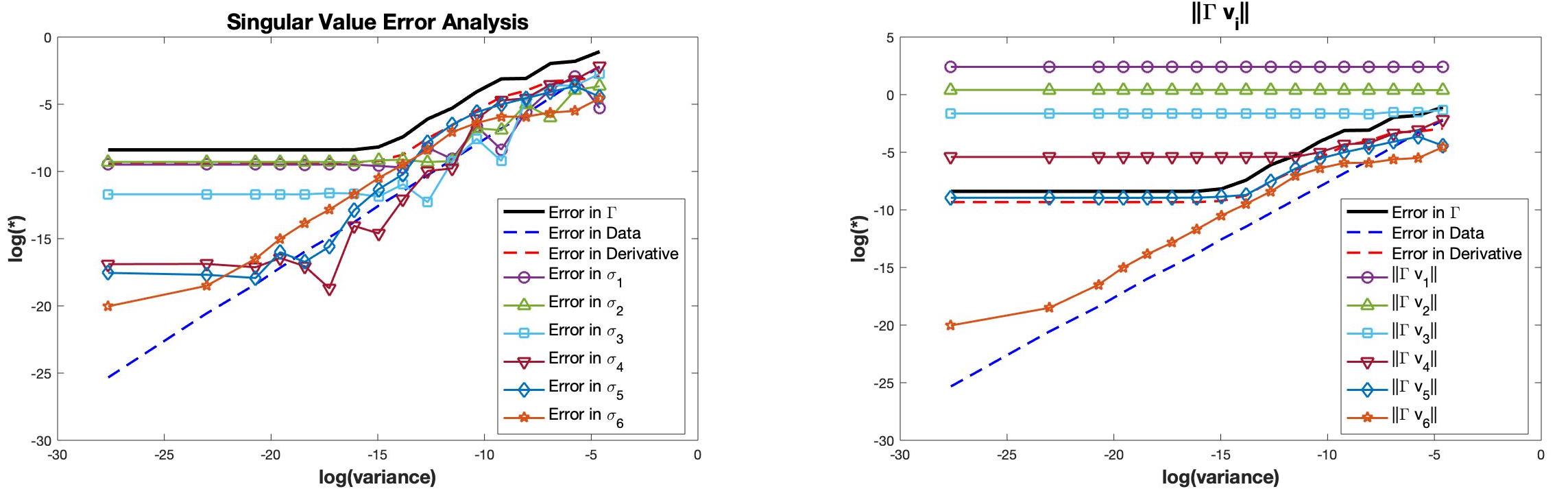}\\
    \normalsize{\textbf{Example 4: No Conservation Law}}
    \includegraphics[width = 0.86 \textwidth]{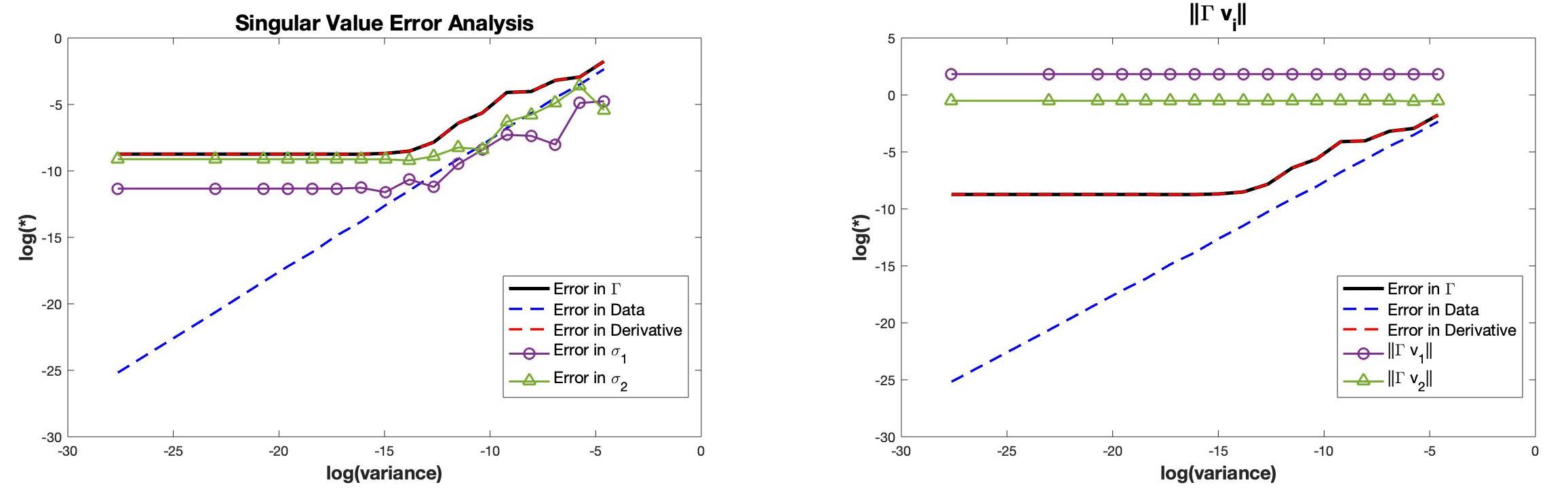}
    \caption{Log-Log plots displaying the effect of noise to the (Left) singular values, $\sigma_i$, and (Right) the corresponding $\lVert \Gamma v_i\rVert$. Each example is computed for the library corresponding to the known conservation law, or the linear library in the example having no conservation. Errors corresponding to $\varepsilon_{\textbf{x}}$,$\varepsilon_{\dot{\textbf{x}}}$, $\mathcal{E}_{\Gamma}$ will be shown for comparison. 
    }
    \label{fig:sing_error_all}
\end{figure}

\begin{corollary}[Weyl-type theorem for $\Gamma$ matrix]
Let $\textbf{X} = [\bx(t_1), \textbf{x}(t_2), ..., \textbf{x}(t_N)]^T \in \mathbb{R}^{N\times n}$ be the time series data for the system with time derivative data given by $\dot{\textbf{X}} = [\dot{\textbf{x}}(t_1), \dot{\textbf{x}}(t_2), \ldots, \dot{\textbf{x}}(t_N)]$. Consider the library matrix $\Theta(\textbf{X})\in \mathbb{R}^{N\times p}$ with $\{\theta_i(\textbf{X})\}_{i=1,\ldots,p}$ and $\theta_i\in \mathbb{P}^k$, where $\mathbb{P}^k$ denotes polynomials of degree $k$, and the corresponding $\Gamma(\textbf{X},\dot{\textbf{X}}) = \dfrac{d}{dt}\Theta(\textbf{X})$. For $\tilde{\Gamma} = \Gamma(\textbf{X}+\varepsilon_\textbf{x},\dot{\textbf{X}}+\varepsilon_{\dot{\textbf{x}}})=\Gamma + \mathcal{E}_\Gamma$, then the following estimate holds 
$$\|\tilde \sigma_i -\sigma_i \|_2 \le  \lVert \mathcal{E}_{\Gamma} \rVert_2 \le \sqrt{Np}\, \mathrm{max}(\|\varepsilon_x\|_\infty,\|\varepsilon_{\dot x}\|_\infty)$$ 
\end{corollary}
\begin{proof}
Notice that since $|\mathcal{E}_{{\Gamma}_{i,j}}|$ consists of $\bx, \dot \bx$ and combinations of thereof, we get $|\mathcal{E}_{{\Gamma}_{i,j}}|\le \mathrm{max}(\|\varepsilon_x\|_\infty,\|\varepsilon_{\dot x}\|_\infty)$ and the result then follows directly from Weyl's Theorem and norm equivalence since $\lVert \mathcal{E}_{\Gamma} \rVert_2\le \sqrt{Np}\lVert \mathcal{E}_{\Gamma} \rVert_{max}$.  
\end{proof}

Note that the bound above is an overestimation of the error and is not expected to be sharp. While it serves the purpose of elucidating the approximate trend of the error, this result requires knowledge of the error in $\dot{\textbf{x}}$, which is dependent on the numerical differentiation method. 

The following corollary gives a precise bound in the case of Tikhonov regularization.

\begin{corollary}
If Tikhonov regularization is used for derivative estimation and the interval is rescaled to $[0,1]$, so that $t_i=ih$ with time spacing $h=1/N$, then by the result in \cite{Shuai06}
$$\varepsilon_{\dot x} \le 8 k_1 h^2 + 2 k^{1/3}_1 k_2^{2/3}  ||\varepsilon_{x}||^{2/3}_\infty$$
where $k_1=64(1+2||f^{(3)}||_2)$ and $k_2=\sqrt{4+2||f^{(3)}||^2_2}$ where $f(\mathrm{x})$ is the dynamics of the equation $\dot {\mathrm{x}} = f(\mathrm{x})$ being approximated. It follows that
\begin{equation}
\| \tilde \sigma_i - \sigma_i\|\le \lVert \mathcal{E}_{\Gamma} \rVert_2
\le \sqrt{Np}\, (k_1 h^2 + 2 k_2||\varepsilon_{x}||^{2/3}_\infty)
\label{eq:gamma_bound}
\end{equation}
\label{thm: sigma_bound}
\end{corollary}

Through these results, we see an interplay between the influence of discretization error and noise. For very small values of added noise, the error bound for $\Gamma$ is dominated by $\sqrt{Np}C_1h^2$, while as we increase the error in $\textbf{x}$ to values larger than $h^2$, we begin to see more of the effect of noise on $\Gamma$. This trend is confirmed by our numerical experiments in Figure \ref{fig:sing_error_all}, where the solid black line corresponds to the error in $\Gamma$. We observe that the error remains constant until $\|\varepsilon_x\|^{2/3} \approx h^2$, at which point the error begins to increase. This will serve as a guiding principle for picking the number of measurements needed for desired accuracy.

For recovering the conservation law, we are primarily interested in the right singular vectors corresponding to smallest singular values. It's important to ensure that we can rely on reasonable perturbation bounds in this process. Numerically, in Figure \ref{fig:sing_error_all} we see stable behavior of the residual $\|\Gamma v_i\|$ w.r.t. to noise. For the  $v_i$ corresponding to the larger singular values, there is very little change in $\|\Gamma v_i\|$, while for $\sigma_j<\sigma_{\textrm{cutoff}}$, $\|\Gamma v_i\|$ follows the error in $\sigma_j$, but is not necessarily bounded by $\mathcal{E}_{\Gamma}$.
Welin's theorem is often used in this context to establish an upper bound on the perturbation of singular subspaces \cite{Wedin}. Its drawback is that it combines left and right singular vectors under one umbrella overestimates the error in each of them in cases of ``tall'' or ``long skinny'' matrices. In our case, since the calculation concerns right singular vectors, the following analytical result proves more useful and provides insight into the numerical behavior.

\begin{theorem}[\cite{Cai}]
Let $\Gamma = U \Sigma V^T\in\mathbb{R}^{N\times p}$ is of rank $r$ and $\tilde \Gamma = \tilde U \tilde \Sigma \tilde V^T=\Gamma + \mathcal{E}_\Gamma$ is its perturbation with i.i.d. zero-mean Gaussian noise $\mathcal{E}$ having unit variance. More precisely, assume that for some $\tau>0$, $\mathbb{E}\exp(t\mathcal{E}_\Gamma)\le \exp(\tau t)$ for all $t\in\mathbb{R}$. 
Let singular values be $\sigma_1\ge \ldots \sigma_r\ge 0$ and 
define principle angles as 
$$
 \Theta (V,\tilde V) = diag(\cos^{-1}(\sigma_1), \ldots, \cos^{-1}(\sigma_r) )
$$
then 
for the distances $\|\sin(\Theta(V,\tilde V))\|$ and  $\|\sin(\Theta(U,\tilde U))\|$ between 
right and left singular subspaces   there exists $C>0$ that only depends on $\tau$ such that 
$$
\mathbb{E}\|\sin(\Theta(V,\tilde V)\|^2 \le \min\left(\frac{Cp(\sigma^2_r+N)}{\sigma^4_r},1\right)
$$
$$
\mathbb{E}\|\sin(\Theta(U,\tilde U)\|^2 \le \min\left(\frac{CN(\sigma^2_r+p)}{\sigma^4_r},1\right)
$$
\label{thm: vec_bound}
\end{theorem}

The authors of \cite{Cai} also provide lower bounds for the subspace distances which demonstrate sharpness of these perturbation results, with a noteworthy difference between the bounds for left and right singular subspaces. These bounds are non-trivial if there exists a constant $C_{gap}>0$ such that
\begin{equation}
\sigma^2_r\ge C_{gap} (\sqrt{Np}+N)
\label{gap_cond}
\end{equation}
This conforms with the classical result about singular vector perturbation behaving as $\|\mathcal{E}_\Gamma\|_2\sigma^{-1}_r$ \cite{Demmel}. However, the explicit dependence on the matrix dimensions makes it more practically useful in our context. In addition, spectral properties of $\Gamma$ play an important role since for small $\sigma_{r}$ reliable recovery will become impossible. 
 
In developing our computational framework, we will use the following guiding principles based on this theoretical analysis:
\begin{itemize}
\item The number of observations $N$ needs to significantly exceed the number of library functions $p$ in $\Gamma$. If this is the case, Theorem \ref{thm: vec_bound} assures that right singular vectors can be consistently recovered through solving the SVD problem.
\vspace{-.1in}
\item If some information about the noise and accuracy tolerance is available, $N$ and $p$ may be chosen to satisfy $\| \sin(\Theta(V,\tilde V)) \|<tol$ using the precise accuracy bound given in Theorem \ref{thm: vec_bound}, with a caveat being the difficulty in estimating the $C$ value.
\vspace{-.1in}
\item For a given library and associated spectral properties of $\Gamma$, one may calculate $\sigma_p$ value and ensure $N$ and $p$ satisfy the necessary condition given in \eqref{gap_cond}. 
\vspace{-.1in}
\item Condition \eqref{eq:gamma_bound} can be used to identify singular values associated with null space of $\Gamma$. 
\vspace{-.1in}
\item One needs to be aware that there exists no stable algorithm for recovering left singular vectors in the context of additive perturbation of this kind. 
\end{itemize}

In summary, while SVD-based methodology imposes certain conditions to ensure numerical accuracy and stability guarantees, in most cases they can be achieved by adding a modest number of data points and reducing the list of candidate functions. We are now ready to formulate the computational framework allowing to automate this process in search for hidden conservation laws in a given dataset.
 

 \section{Computational framework for learning conservation laws}  
 \label{sec:identify}
Now we will develop an algorithm for identifying the optimal $\Theta$-library from a set of possible libraries. The goal is to find a library that minimizes \eqref{eq:minimization} while not including extraneous terms in the conservation law. The found conservation law will be compared to the known solution and the conserved quantity in Equation \eqref{eq:findcons} will be recorded.

Let us consider data given by $\textbf{X}\in \mathbb{R}^{N\times n}$ and a set of distinct $\Theta$-libraries, $\Phi = \{\Theta^{(1)},\ldots, \Theta^{(k)}\}$ where $\Theta^{(i)}\in \mathbb{R}^{p_i}$. We wish to identify which $\Theta^{(i)}$ is optimal, if such a $\Theta$ exists. For each $\Theta^{(i)}$, we will run the analysis shown in Figure \ref{fig:flowchart1}, paying special attention to all the singular values in each case. Each $\Theta^{(i)}$-library will have a corresponding $\Gamma^{(i)}$-library and it's SVD, $U^{(i)}S^{(i)}(V^{(i)})^T$.  

Each $S^{(i)}$ contains the singular values $\sigma^{(i)}_{1}\leq \ldots \leq \sigma^{(i)}_{p_i}$ on the diagonal.  Let $j$ denote the index of the first singular value below a predefined cutoff such that $\sigma_j^{(i)}, \ldots, \sigma_{p_i}^{(i)} < \sigma_{\textrm{cutoff}}$. Let 
\begin{equation}
\textrm{count}^{(i)} = \textrm{length}\begin{bmatrix} \sigma_j^{(i)} & \ldots & \sigma_{p_i}^{(i)} \end{bmatrix} =
\begin{cases}
0 & \textrm{if no } \sigma_j<\sigma_{\textrm{cutoff}}\\
p_i - j+1 &  \textrm{otherwise}
\end{cases}.
\label{eq:count}
\end{equation}
If $\textrm{count}^{(i)} = 0$, then the choice of library does not produce a conservation law and no further analysis can be done for this choice of $\Theta$-library. Furthermore, if $\textrm{count}^{(i)} >n$, we will also discard these $\Theta^{(i)}$ as they have produced more conservation laws than states in the system. This could indicate that the cutoff value is placed too high or there is a library containing a subset of the functions which is more optimal. 

\begin{remark}
    The cutoff value will vary model to model. If prior knowledge about the noise in the system is known, the cutoff may be estimated used the results in Section \ref{sec:noise}.
    Based on Corollary \ref{thm: sigma_bound}, we suggest the following rule for choosing the cutoff value: $\sigma_{\textrm{cutoff}} =\sqrt{Np}\|\varepsilon_x\|_{\textrm{max}}^{2/3}$, changing according to the  library considered. This bound follows \eqref{eq:gamma_bound} while still obeying Equation \eqref{eq:sig_bounds}.
\end{remark}

Of the remaining libraries, we will look closer at their corresponding singular values. In practice, larger singular values, $\begin{bmatrix} \sigma_1^{(i)} & \ldots & \sigma_{j-1}^{(i)} \end{bmatrix}$ are often retained in order to produce matrix approximations for large matrices as those are the largest contributors of information in the transformed system \cite{BruntonKutz}. In a similar sense, we are obtaining our null space approximations from the smallest singular values, $\begin{bmatrix} \sigma_j^{(i)} & \ldots & \sigma_{p_i}^{(i)} \end{bmatrix}$, as they contribute the ``least" amount of information. We are interested in the gap between these two sets, defined by  
\begin{equation}
    \delta^{(i)} = \sigma_{j-1}^{(i)} - \sigma_{j}^{(i)}
    \label{eq:sigmagap}
\end{equation}
Each $\delta^{(i)}$ measures the discrepancy between the singular values which approximate the matrix and those which contribute to the null space. Optimal libraries will have a large $\delta^{(i)}$, indicating that there is a clear distinction between the two sets. Small $\delta^{(i)}$ indicate that $\sigma_{j-1}^{(i)}$ could have potentially been included in the other set had the threshold value allowed it. As we will see in the examples, most of these significantly small $\delta^{(i)}$'s correspond to libraries which have a large $p_i$. 

\begin{remark}
    While not a sufficient criteria for dismissal, one can also check $\Gamma^{(i)} v_j^{(i)} \leq \sigma_{\textrm{cutoff}}$ for all $j$. This will reaffirm that the choice of library does indeed produce a conservation law for the system. This is not always the best indicator of an optimal library as small coefficients can be included in libraries as the algorithm attempts to ``tune" the result to the system. 
\end{remark}

Once the optimal library has been chosen,  all that remains is to record the corresponding right singular vector(s), $v_j$, for the $\sigma_j$'s. The $v_j$ contain the coefficients for the terms in the chosen $\Theta$-library. Flowchart \ref{fig:choosetheta} details this process and Algorithm \ref{algor:opttheta} provides associated pseudo-code.

\begin{figure}[h]
\begin{center}
\includegraphics[width = 1.0\textwidth]{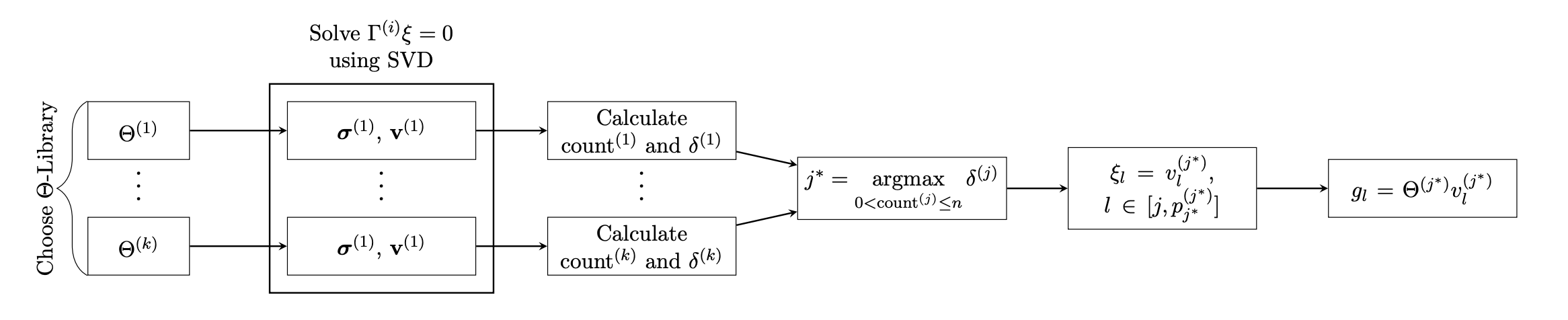}
\caption{Flowchart detailing the process for selecting the optimal $\Theta$-library. For each potential library, singular values and vectors will be recorded and count and $\delta$  are calculated. The optimal $\Theta$ will correspond to the system with $\textrm{count} \in (0,n]$ and the largest $\delta$. Once those have been selected, the corresponding singular vector(s) can be defined as the coefficients for the functions in the chosen $\Theta$-library.} 
\label{fig:choosetheta}
\end{center}
\end{figure}

\begin{remark}
Recovered conservation laws can be simplified by finding the reduced row echelon form (RREF), 
possibly using a permutation matrix to ensure the best pivot points are chosen \cite{Demmel}. Choosing to write the recovered laws in this form may provide insight into the underlying network structure as well as eliminating any scaling.
\end{remark}

 \begin{algorithm}[h]
  \caption{Algorithm for finding optimal $\Theta$-library from a list of candidate libraries. 
  } 
  \label{algor:opttheta}
  \begin{algorithmic}
 \Require {$\textbf{X}\in \mathbb{R}^{N\times n}$  - Data, $d\textbf{X} \in \mathbb{R}^{N\times n}$ - derivative data, $\{\Theta^{(i)}\}_{i=1,\ldots,k}$ - chosen libraries to test, cutoff value - $\sigma_{\textrm{cutoff}}$}
 \Ensure $\Theta^{(\textrm{opt})}$, $\bm{\Xi}$
 \State $\Phi = \{ \Theta^{(1)}(\textbf{x}), \ldots, \Theta^{(k)}(\textbf{x}) \} $  \Comment{Library of considered libraries}\;
 \For{$i\in \{1,\ldots, k\}$}
  \State $\Theta^{(i)} = \Phi_i$  \Comment{Choose a new $\Theta$ library}\
  
\State $\Gamma^{(i)} = (D_x\Theta^{(i)})^T\dot{\textbf{x}}(t) = \begin{bmatrix} \nabla_\textbf{x}\theta_1^{(i)}(\textbf{X})\cdot d\textbf{X} &\ldots &\nabla_\textbf{x}\theta_{p_i}^{(i)}(\textbf{X})\cdot d\textbf{X}  \end{bmatrix} $ \Comment{Calculate $\Gamma$}
  \State $\Gamma^{(i)} = U^{(i)}S^{(i)}(V^{(i)})^T$\; 
  \State $S^{(i)} = \textrm{diag}(\sigma^{(i)}_j) \hspace{1em} \sigma_1 \geq \ldots \geq \sigma_{p_i}$\;  
\State $\textrm{ind}^{(i)} = \{j| \sigma_j^{(i)}<\sigma_{\textrm{cutoff}}\}$ \Comment{Find indices where  $\sigma \approx 0$}\;
 \State  $count^{(i)} = \textrm{length}(\textrm{ind}^{(i)})$ \Comment{Count the number of nonzero entries}
  \If{$0 < count^{(i)} < n$}
  \State $\delta^{(i)} = \{ \sigma^{(i)}_{j-1}-\sigma^{(i)}_{j} | j =\min(\textrm{ind}^{(i)}) \}$  \Comment{Sigma-gap}
  \Else
  \State go to next section\Comment{Too many recovered conservation laws or none.} 
  \EndIf
  \EndFor
  \State opt $\gets \argmax\limits_{i}\delta^{(i)}$\;
  \State \textbf{Return:} $\Theta^{(\textrm{opt})}$ \& $\bm{\Xi} = \{V^{(opt)}_j| j\in \textrm{ind}^{(\textrm{opt}))}\}$\;
 \State  \textrm{\textbf{Optional}: } $P\bm{\Xi} = LU \hspace{1em} \rightarrow \hspace{1em} \bm{\Xi }= \textrm{RREF}((P\bm{\Xi})^T)P^{T}$
  \end{algorithmic}
\end{algorithm}

 \section{Numerical Results}
\label{sec:testing}
Now we will implement Algorithm \ref{algor:opttheta} for the examples in Table \ref{tab:examples} for $N=20, 100$ with noise variance $=0,1e-10,1e-5$. The candidate libraries will contain a combination of: monomials up to order 3, trigonometric terms, and logarithmic terms, with the largest library considered being $\begin{bmatrix} \textbf{x} & \textbf{x}^2 & \textbf{x}^3 & \sin(\textbf{x}) & \cos(\textbf{x}) & \ln(\textbf{x}) \end{bmatrix}$.
In the following figures, libraries will be represented by a triple, $(a,b,c)$, corresponding to: polynomial terms up to order $a$, include ($b=1$)/exclude ($b=0$) trigonometric terms, and include ($c=1$)/exclude ($c=0$) logarithmic terms. We will only consider libraries in which $N\geq p$, ensuring accurate null space recovery.  A complete list of libraries considered and their corresponding triple can be found in the Supplemental Notes. 

Data is simulated from a known ODE system for $t=[0:dt:1]$, where $dt = \frac{1}{N}$ and i.i.d. Gaussian noise with a given variance is added. The corresponding derivative data will be generated using numerical differentiation methods refined in \cite{Wagner20}. Tikhonov regularization is used for Examples 1, 2, and 4 whereas the Savitzky-Golay filter is used for Example 2.

\begin{remark}
   In our observations, the ``ideal" number of points needed is dependent on the libraries considered. According to Theorem \ref{thm: vec_bound}, to ensure possibility of proper recovery of singular subspaces, $\Gamma$ should necessarily be rectangular with $N\geq p$.
   For methods to identify a sparse null space in the under-determined case, see \cite{Coleman86,Coleman87}
\end{remark}

For each, we graph all singular values, $\sigma_i$, for each candidate library and the corresponding cutoff (black line). Both the number of $\sigma_i < \sigma_{\textrm{cutoff}}$ as well as the distance between points directly above and below $\sigma_{\textrm{cutoff}}$ will be used to determine the optimal library. In each figure, we shade the (perturbed) singular value(s) corresponding to the true conservation law. Ideally we wish to see the algorithm selecting the libraries with shaded values and only shaded values appearing below the cutoff. Once the optimal library is chosen, we can consider the errors related to the recovered conservation law(s) follow and overall method accuracy for 1000 runs. For more detailed analysis of the overall process, refer to the tables in the Supplemental Notes. 

All figures are generated using Matlab with formatting 
developed in \cite{subtightplot}. 

\subsection{Example 1}
\label{sec:examples-volpert}

Figure \ref{fig:volpert_sing} contains the analysis for Example 1 in Section \ref{sec:examples}.  In all cases,  first order polynomial libraries produced the optimal conservation law. For example, in the case of  20 points with a variance of 0, four libraries, $(1,0,0),(1,0,1),(2,0,0),(1,1,1)$, have exactly one singular value below the cutoff while the remaining libraries contain more than 3. 
Of the four libraries meeting the counting criteria, by looking at $\delta$ as defined in Equation \eqref{eq:sigmagap}, it is clear that the first library corresponding to linear terms has the largest gap. This is shown numerically in the Supplemental Notes. 

\begin{figure}[H]
    \includegraphics[width = 1.0\textwidth]{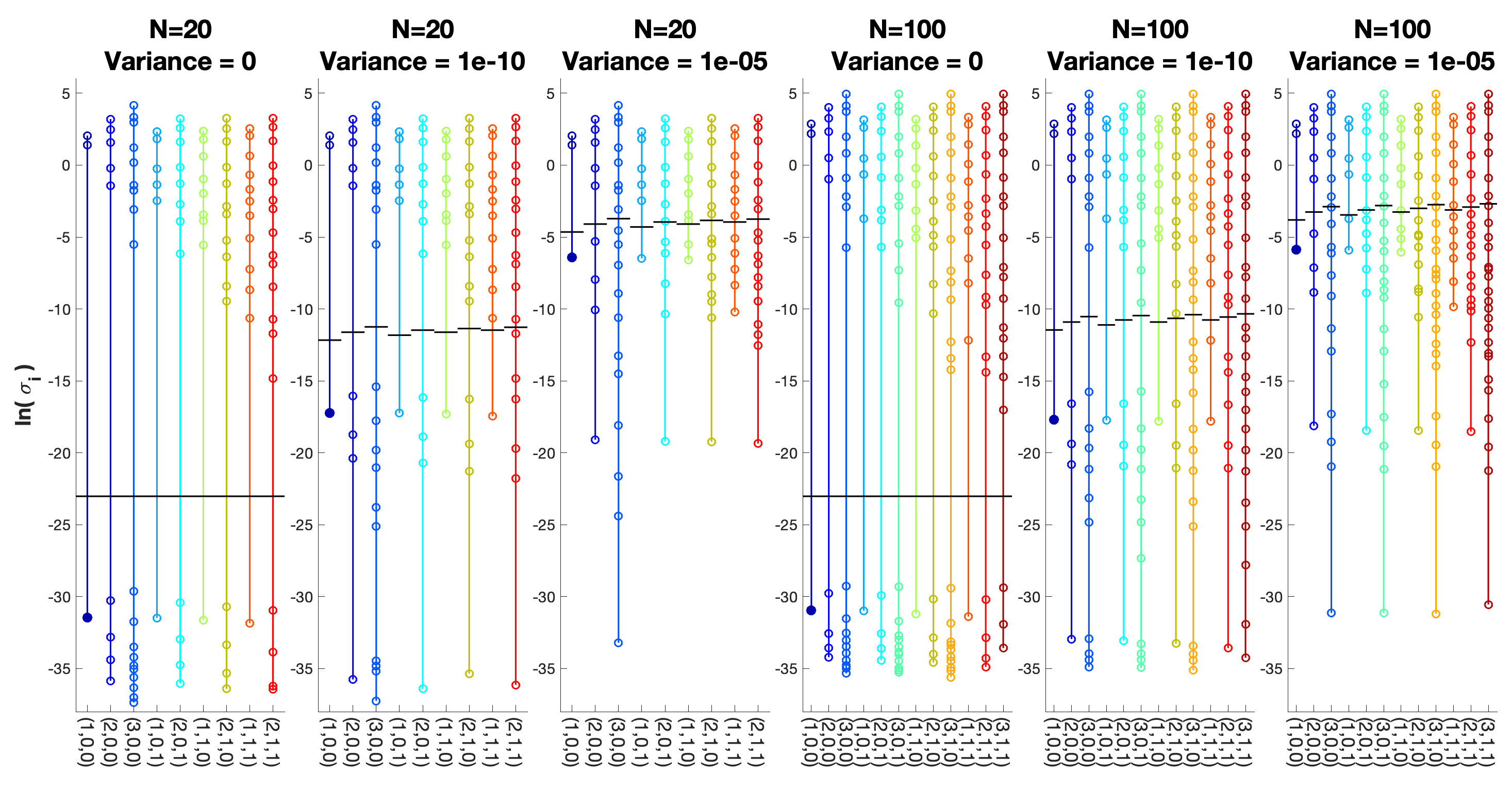}

    \caption{Singular values corresponding to different library configurations for Example 1 with $N$ points and given noise variance. 
    }
    \label{fig:volpert_sing}
\end{figure}

Table \ref{tab:volpert_vecs} contains error values for the data, $\Gamma$, the conserved quantity, and the recovered law, in original and reduced form. As can be seen, as we increase the noise added to the system, we are also increasing the overall error in the conserved quantity and recovered vector, although it remains close to the variance of the noise added.  For example, in the first setup with 20 points and no noise, the recovered conservation law is:
\begin{equation}
\label{eq:volpert_1}
0.5774\dot{x}_1+0.5774\dot{x}_2+0.5774\dot{x}_3 = 0 \hspace{2em} \Rightarrow \hspace{2em} x_1+x_2+x_3 = C
\end{equation}

In all cases, assuming appropriate cutoffs are defined, we were able to achieve an accuracy of 100\% when tested on 1000 randomly generated data sets. 
\begingroup
\begin{table}[H]
    \centering
    \renewcommand{\arraystretch}{1.5}
    \begin{tabular}{|c|c|c|c|c|c|c|c|}\hline
    N & $\|\varepsilon_x\|$ & $\lVert \varepsilon_{\dot{\textbf{x}}}\rVert  $& $\lVert \mathcal{E}_{\Gamma}\rVert  $& $\sum\limits_{j}\lVert \tilde{\Gamma}\xi_{\textrm{opt},j}\rVert$ & $\sum\limits_{j}\lVert \tilde{\Gamma}\xi_{\textrm{rref},j}\rVert$  & $\lVert \xi_{\textrm{exact}}-\xi_{\textrm{rref}}\rVert$& Accuracy \\\hline
         \multirow {3}{4em}{20 points}& 0 &0.1836 & 0.1836 & 2.179e-14 & 3.746e-14 & 3.144e-15 & 100\%\\\cline{2-8}
        & 5.4849e-10 & 0.1836 &0.1836 & 3.326e-8 & 5.761e-8 & 1.207e-9& 100\%\\ \cline{2-8}
         & 4.3606e-5 & 0.1834 & 0.1834 & 1.608e-3 & 2.785e-3 & 1.390e-4& 100\%\\ \hline
         \multirow {3}{4em}{100 points} & 0 & 0.05127 & 0.05127 & 3.538e-14 & 6.144e-14 & 1.351e-15& 100\%\\ \cline{2-8}
        & 4.7594e-10 & 0.05127 & 0.5127 & 2.021e-8 & 3.500e-8 & 9.294e-10& 100\%\\ \cline{2-8}
         & 4.579e-5 & 0.0513 & 0.0513 & 2.774e-3 & 4.805e-3 & 1.019e-4 & 100\% \\ \hline
    \end{tabular}
    \caption{Table of errors related to the found conservation law for Example 1.}
    \label{tab:volpert_vecs}
\end{table}
\endgroup

\subsection{Example 2}
\label{sec:examples-twocons}
In Figure \ref{fig:twocons_sing} we see the optimal library will again contain only first order polynomials, however, we now have two $\sigma_i<\sigma_{\textrm{cutoff}}$, indicating we have two recovered conservation laws. It is interesting to note that two distinct conservation laws are recovered rather than one combined law.

\begin{figure}[h]
 \includegraphics[width = 1.0\textwidth]{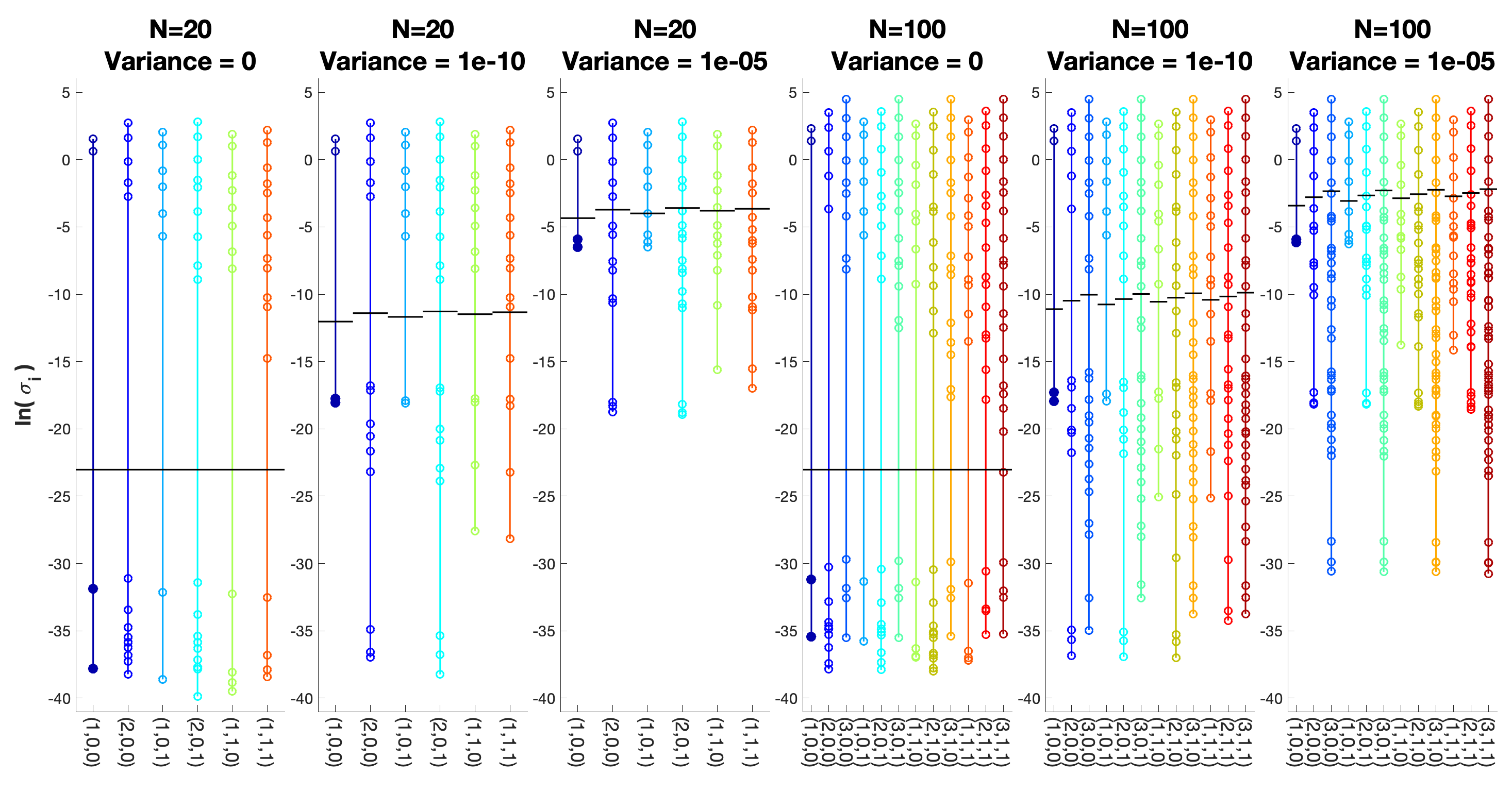}

    \caption{Singular values corresponding to different library configurations for Example 2 with $N$ points and given noise variance. 
    }
    \label{fig:twocons_sing}
\end{figure}
 Despite needing to recover two conservation laws, Table \ref{tab:twocons_vecs} contains a similar trend occurring between $\varepsilon_{\textbf{x}}$ and the error incurred in the recovered law and conserved quantity. As before, we recovered the correct form of the conservation law 100\% of the time for 1000 randomly generated test sets.

\begingroup
\begin{table}[H]
    \centering
    \renewcommand{\arraystretch}{1.5}
    \begin{tabular}{|c|c|c|c|c|c|c|c|}\hline
    N & $\|\varepsilon_x\|$ & $\lVert \varepsilon_{\dot{\textbf{x}}}\rVert  $& $\lVert \mathcal{E}_{\Gamma}\rVert  $& $\sum
    \limits_{j}\lVert \tilde{\Gamma}\xi_{\textrm{opt},j}\rVert$ & $\sum\limits_{j}\lVert \tilde{\Gamma}\xi_{\textrm{rref},j}\rVert$  & $\lVert \xi_{\textrm{exact}}-\xi_{\textrm{rref}}\rVert$& Accuracy \\\hline
        \multirow {3}{4em}{20 points}& 0 & 0.32781 & 0.32781 & 1.472e-14 & 2.322e-14 & 6.026e-15& 100\%\\ \cline{2-8}
        & 5.4588e-10& 0.3278 & 0.3278 & 3.402e-8 & 5.453e-8 & 2.065e-9& 100\%\\ \cline{2-8} 
        & 5.5253e-5 & 0.3280& 0.3280 & 4.199e-3 & 7.029e-3 & 4.432e-4& 100\%\\ \hline
        \multirow {3}{4em}{100 points} &0 & 0.1045 & 0.1045 & 3.034e-14 & 4.754e-14 & 5.989e-15& 100\%\\ \cline{2-8}
        & 6.522e-10 &0.1045 & 0.1045 & 4.740e-8 & 9.034e-8 & 1.706e-9& 100\%\\ \cline{2-8}
        & 6.6525e-5 & 0.1049 & 0.1049 & 4.733e-3 & 7.136e-3 & 1.760e-4 & 100\%\\ \hline
    \end{tabular}
    \caption{Table of  errors corresponding to the found conservation law for Example 2.}
    \label{tab:twocons_vecs}
\end{table}
\endgroup

\subsection{Example 3}
\label{sec:examples-chemcial}
Our previous examples show we can recover both single and multiple conservation laws, however, in each case, the recovered law was linear. In Example 3, we will consider a system which contains a nonlinear conservation law. In almost all the cases shown in Figure \ref{fig:chemical_sing}, libraries containing both first order polynomials and logarithmic terms are optimal. The exception to this is when $N=100$, and variance was $1e-5$. In this case, we recover the correct form, however two laws are obtained. It is also worth note that in all cases libraries containing only first order polynomials failed to have any $\sigma_i<\sigma_{\textrm{cutoff}}$.

\begin{figure}[h]
     \includegraphics[width = 1.0\textwidth]{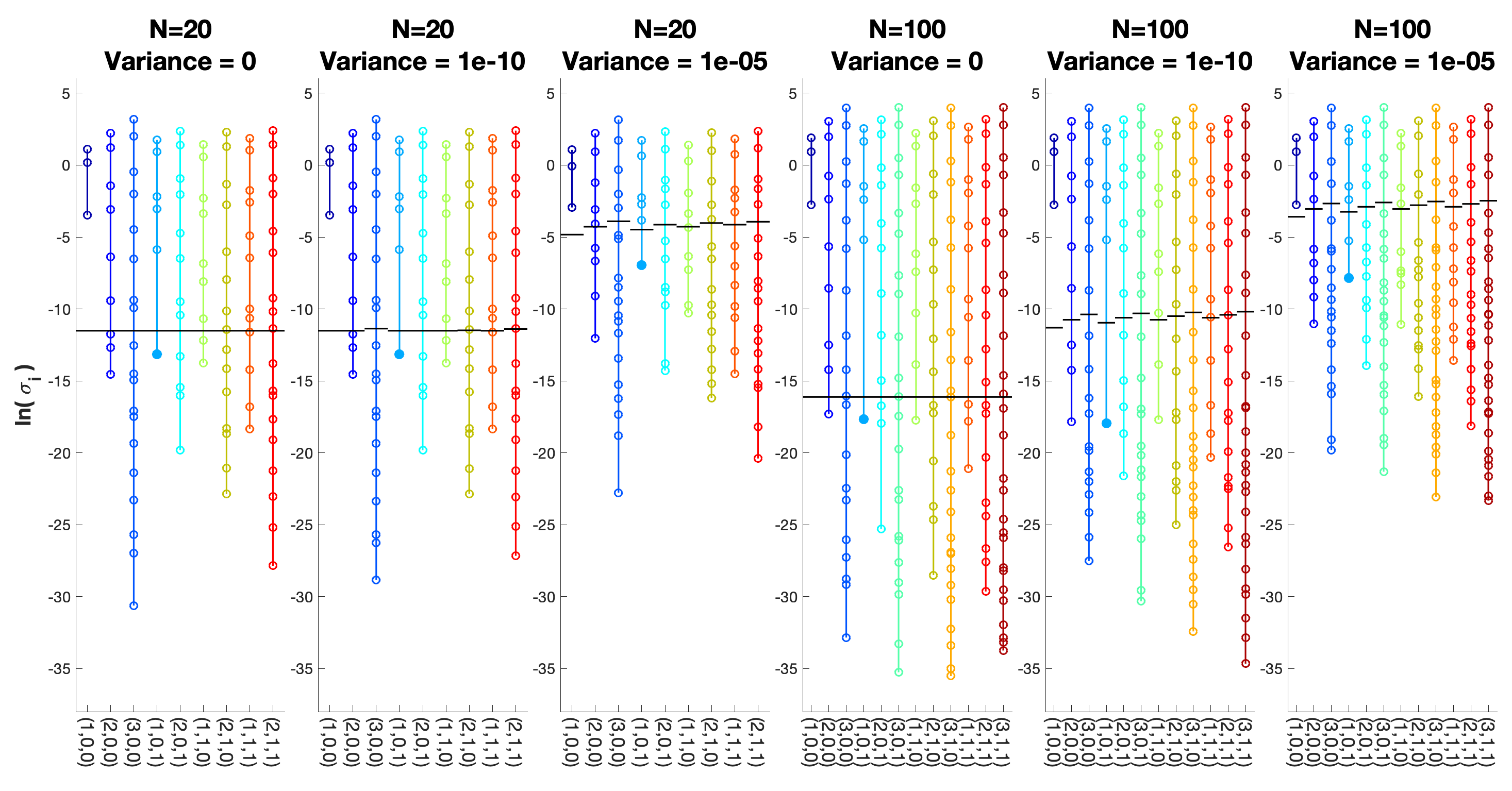}
    \caption{Singular values corresponding to different library configurations for Example 3 with $N$ points and given noise variance.
    }
    \label{fig:chemical_sing}
\end{figure}
For this system, noise from both the derivatives as well as the data play a part in the noise in $\Gamma$. Regardless of the increase in sources of error for $\Gamma$, Table \ref{tab:chemical_vecs} maintains the relationship between $\mathcal{E}_{\Gamma}$ and the conserved quantity. Furthermore, we were able to achieve the optimal library of the conservation law over 98\% of the time for 1000 randomly generated data sets. Further refinement of the derivative calculations could see an increase in the overall method accuracy. 

\begingroup
\begin{table}[H]
    \centering
    \renewcommand{\arraystretch}{1.4}
    \begin{tabular}{|c|c|c|c|c|c|c|c|}\hline
     N & $\|\varepsilon_x\|$ & $\lVert \varepsilon_{\dot{\textbf{x}}}\rVert  $& $\lVert \mathcal{E}_{\Gamma}\rVert  $& $\sum
    \limits_{j}\lVert \tilde{\Gamma}\xi_{\textrm{opt},j}\rVert$ & $\sum\limits_{j}\lVert \tilde{\Gamma}\xi_{\textrm{rref},j}\rVert$  & $\lVert \xi_{\textrm{exact}}-\xi_{\textrm{rref}}\rVert$& Accuracy \\\hline
        \multirow {3}{4em}{20 points} & 0 & 1.1899e-4 & 2.811e-4 & 1.974e-6 & 2.269e-6& 6.946e-4 & 100\%\\ \cline{2-8}
        & 4.5634e-10 & 1.1899e-4 & 2.811e-4 & 1.973e-6 & 2.267e-6 & 6.950e-4 & 100\% \\\cline{2-8}
        & 3.310e-5 & 5.4656e-1 & 1.117 & 9.561e-4 & 1.026e-3 & 0.6206 & 98.6\%\\ \hline
        \multirow {3}{4em}{100 points}& 0 & 4.4893e-8 & 1.674e-7 & 2.141e-8 & 2.460e-8 & 1.072e-6 & 100\%\\ \cline{2-8}
        & 6.0004e-10 & 7.5211e-8 & 1.477e-7 & 1.594e-8 & 1.831e-8 & 6.129e-7 & 99.9\%\\ \cline{2-8}
        & 6.3398e-5 & 2.0047e-3 & 5.553e-3 & 5.547e-3 & 9.219e-3 & NaN & 100\%\\ \hline
    \end{tabular}
    \caption{Table of errors corresponding to the found conservation law(s) for Example 3. For the case of $N=100$ and $\|\varepsilon_{\textbf{x}}\| = 6.3398e-5$, we obtain two conservation laws rather than one, therefore no direct comparison can be made between the exact form and the recovered laws.}
    \label{tab:chemical_vecs}
\end{table}
\endgroup

\subsection{Example 4}
\label{sec:examples-nocons}
Our final benchmark example contains no conservation law, testing if the algorithm will output no optimal library. In Figure \ref{fig:noncons_sing}, we expect to see all singular values above the designated cutoff, but this is not the case for examples containing noise. While the algorithm may emit a conservation law, it is important to examine the results and verify they correlate with the system. In cases where the algorithm selected a library the corresponding $\bm{\xi}$ is both dense and associated with large libraries. 
To avoid this, either decrease $\sigma_{\textrm{cutoff}}$ or exclude large libraries that can be over-fitted. 

\begin{figure}[h]
\begin{center}
     \includegraphics[width = 0.96\textwidth]{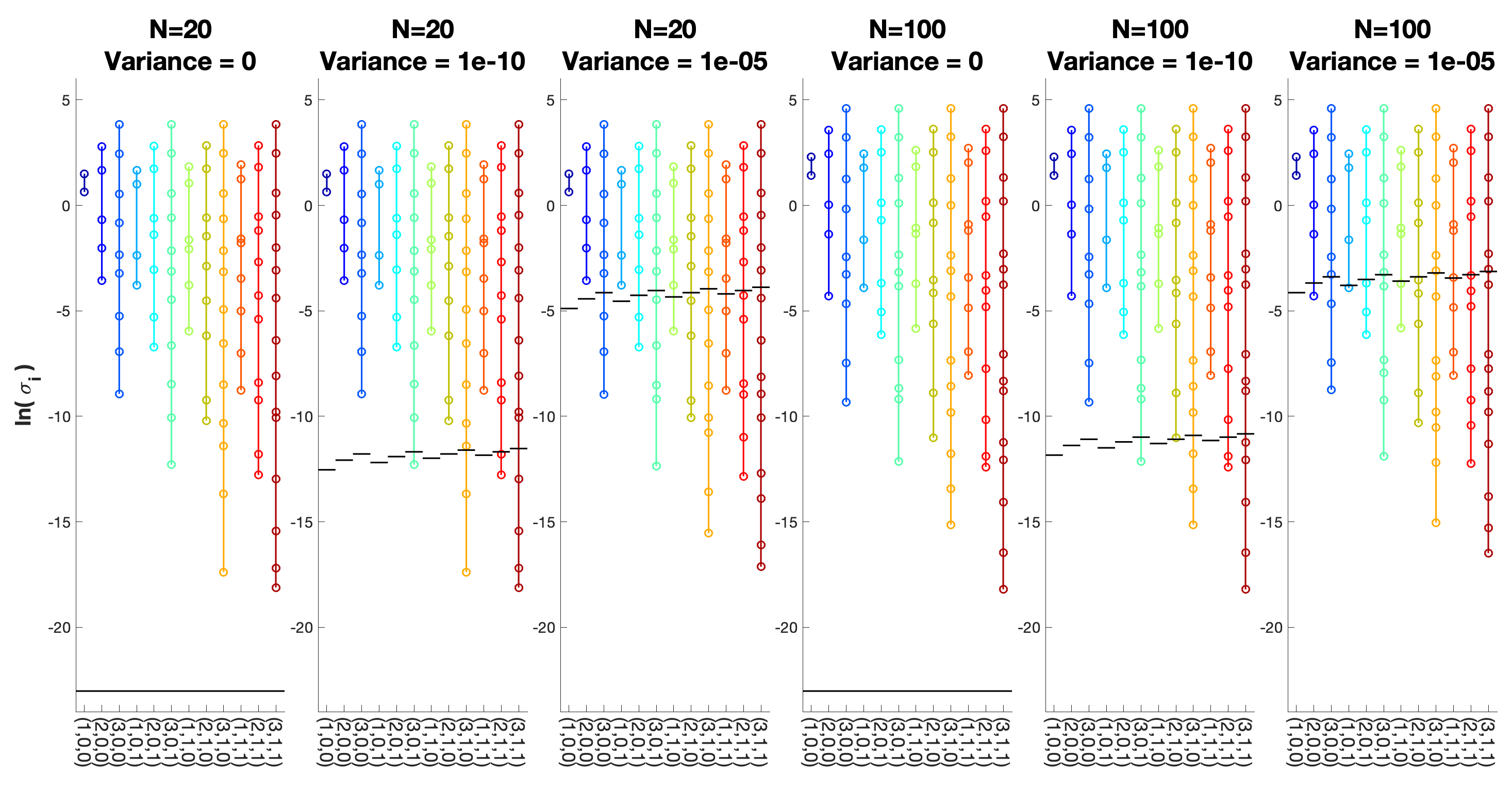}
    \caption{Singular values corresponding to different library configurations for Example 4 with $N$ points and given noise variance.
    }
    \label{fig:noncons_sing}
    \end{center}
\end{figure}

In Table \ref{tab:nocons_vecs}, we have outlined the results of the algorithm when applied to Example 4. In the case of noisy data, one needs to examine the output to verify it is indeed a conservation law. In this case, all the ``found conservation  laws" are dense and are a byproduct of the minimization process.  

\begingroup
\begin{table}[H]
    \centering
    \renewcommand{\arraystretch}{1.25}
    \begin{tabular}{|c|c|c|c|c|c|c|}\hline
     N & $\|\varepsilon_x\|$ & $\lVert \varepsilon_{\dot{\textbf{x}}}\rVert  $& Optimal Library&  count & $\sum
    \limits_{j}\lVert \tilde{\Gamma}\xi_{\textrm{opt},j}\rVert$ & $\sum\limits_{j}\lVert \tilde{\Gamma}\xi_{\textrm{rref},j}\rVert$ \\\hline
        \multirow {3}{4em}{20 points} & 0 & 0.26317 & Not Found & NaN & NaN & NaN \\ \cline{2-7}
        & 4.5288e-10 & 0.26317 & (2, 1, 1) & 2 & 1.024e-5 & 1.270e-5 \\\cline{2-7}
        & 4.1076e-5 & 0.26315 & (2, 0, 1)& 2 & 6.205e-3 & 8.672e-3 \\ \hline
        \multirow {3}{4em}{100 points}& 0 & 0.0761 & Not Found & NaN & NaN & NaN\\ \cline{2-7}
        & 3.6436e-10 & 0.07612 & (3, 0, 1) & 1 & 5.290e-6 & 6.484e-6\\ \cline{2-7}
        & 3.7872e-5 & 0.07625 & (2, 0, 0) & 1 & 0.01377 & 0.01718 \\ \hline
    \end{tabular}
    \caption{Table of errors corresponding to the found conservation law(s) for Example 4. }
    \label{tab:nocons_vecs}
\end{table}
\endgroup

\subsection{Biological Application}
\label{sec:examples-mapk}
While we have shown we can recover various forms of conservation, all benchmark examples utilize mass-action kinetics. To show our approach is independent of the system dynamics let us consider instead a system with Michaelis–Menten kinetics, common in biological systems. We consider the Mitogen Activated Protein Kinase (MAPK) Pathway presented in \cite{Kholodenko10} containing 9 state variables (3 proteins each with 3 states of phosphorylation) and emits 3 conservation laws. The dynamical system representation and parameters used will be included in the Supplemental Notes. Data is generated over $t=[0:dt:1000]$ with $dt = \frac{1}{N}$. As before, i.i.d. Gaussian noise will be added and derivatives are approximated with Tikhonov regularization. Singular values for libraries considered are shown in Figure \ref{fig:kholodenko_sing}. In all cases, the algorithm recovers 3 linear conservation laws.

\begin{figure}[H]
     \centering
 \includegraphics[width = 0.97\textwidth]{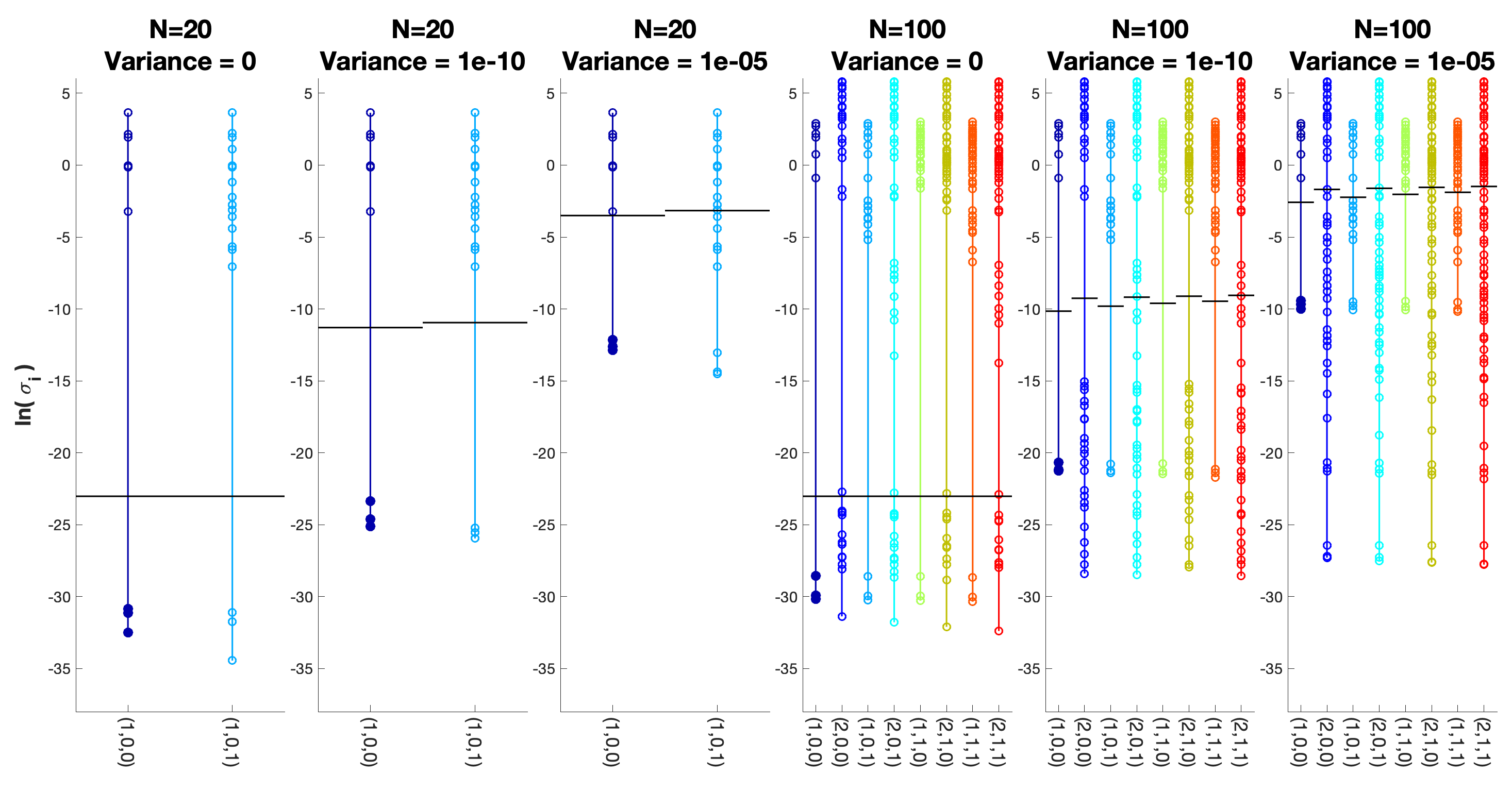}

    \caption{Singular values corresponding to different library configurations for the MAPK model \cite{Kholodenko10} with $N$ points and given noise variance. 
    }
    \label{fig:kholodenko_sing}
\end{figure}
Table \ref{tab:kholodenko_vecs} provides error analysis for the recovered conservation laws. While the overall error in $\Gamma$ has increased, most likely due to the change in scale of the data, the results are consistent with the previous examples. In all cases, we recovered the correct library 100\% of the time.
\begingroup
\begin{table}[H]
    \centering
    \renewcommand{\arraystretch}{1.5}
    \begin{tabular}{|c|c|c|c|c|c|c|c|}\hline
    N & $\|\varepsilon_x\|$ & $\lVert \varepsilon_{\dot{\textbf{x}}}\rVert  $& $\lVert \mathcal{E}_{\Gamma}\rVert  $& $\sum
    \limits_{j}\lVert \tilde{\Gamma}\xi_{\textrm{opt},j}\rVert$ & $\sum\limits_{j}\lVert \tilde{\Gamma}\xi_{\textrm{rref},j}\rVert$  & $\lVert \xi_{\textrm{exact}}-\xi_{\textrm{rref}}\rVert$& Accuracy \\\hline
         \multirow {3}{4em}{20 points} & 0 & 37.31 & 37.31 & 7.828e-14 & 1.383e-13 & 7.057e-13 & 100\%\\ \cline{2-8}
         & 8.8961e-10 & 37.31 & 37.31 & 1.035e-10 & 1.962e-10 & 1.124e-9& 100\%\\\cline{2-8}
        & 1.0553e-4 & 37.31 & 37.31 & 1.114e-5 & 1.944e-5 & 7.386e-5& 100\%\\ \hline
        \multirow {3}{4em}{100 points} & 0 & 4.052 & 4.052 & 5.838e-13 & 1.013e-12 & 4.588e-13& 100\%\\ \cline{2-8}
        & 1.4759e-9 & 4.052 & 4.052 & 2.255e-9 & 4.001e-9 & 9.180e-10& 100\%\\ \cline{2-8}
            & 1.2644e-4 & 4.052 & 4.052 & 1.915e-4 & 3.356e-4 & 9.087e-5& 100\% \\ \hline
    \end{tabular}
    \caption{Table of errors corresponding to the found conservation laws for the MAPK pathway in \cite{Kholodenko10}.}
    \label{tab:kholodenko_vecs}
\end{table}
\endgroup

\section{Discussion}
\label{sec:discussion}
In this work, we have presented an algorithm which can be used to infer conservation laws from data, including the case of limited noisy data. We tested our approach on several systems ranging in complexity and found being useful for both linear and nonlinear dynamical systems, allowing us to identify conservation laws regardless of the underlying structure.  We have shown we can recover both nonlinear conservation in Example 3 as well as multiple conservation laws as demonstrated in Example 2 and the MAPK network. 

Robustness of the proposed methodology is based on the provable  stability of the singular values and singular vectors to the level of noise present in the data. While there are general results of this nature in the literature \cite{Demmel,Shuai06}, the exact error characterizations depend on the type of the matrix used. Our numerical investigations confirm that low sensitivity is preserved in the case of library matrix $\Gamma(\bX, \dot \bX)$ used in the context of present application. In particular, the error $\varepsilon_{\sigma_i}$ for the smallest singular values of $\Gamma$ is tightly bounded above by derivative estimation error. Since this error may be significant in some cases, the accuracy of the derivative approximation plays an essential role in this discussion. A smoothing approach based on Tikhonov regularization has shown potential to address  this issue based on our numerical experiments.

Overall, guided by our analytical and numerical observations, we developed the following strategy for choosing an appropriate $\sigma_{\textrm{cutoff}}$. As seen in Section \ref{sec:noise}, if the underlying system noise can be estimated a priori, then $\sigma_{\textrm{cutoff}}$ will be proportional to the error, although it may require tuning. Furthermore, provided an approximate error in the derivative calculation, $\mathcal{E}_{\Gamma}$ can be estimated and $\sigma_{\textrm{cutoff}}$ can be chosen such that it lies below $\mathcal{E}_{\Gamma}$. If the error in the data is not accessible, then possible solutions include pre-processing the data to remove noise, finding more accurate derivative approximations, or methods of de-noising $\Gamma$ such as those presented in \cite{Epps19}. We can also consider optimizing $\sigma_{\textrm{cutoff}}$ or finding alternative methods to remove it entirely. 

In addition, we see that while the use of RREF and LU decomposition allows for significant improvement in our ability to find the reduced form of the functional relationship, for some choices of the libraries it may result in fragmentation of the recovered conservation laws. 

 While we restricted library options to 12 possible forms, different systems may benefit from additional library options. In this work, each $\Theta$-library is constructed by choosing classes of functions, such as $\textbf{x}$ or $\textbf{x}^3$ and iterates though subsets of the chosen classes, although this limitation is not required for Algorithm \ref{algor:opttheta} to function properly. 
 For systems containing multiple laws, not necessarily all of the same order, allowing libraries to contain only a subset of proposed classes proves to be a more effective strategy. For example, consider the 3-species Lotka-Volterra model presented in \cite{Schimming03} containing both a first order polynomial and a third order polynomial conservation law..
\begin{equation*}
    \begin{aligned}
        &\dot{x}_1 = x_1(x_2-x_3)&\\
        &\dot{x}_2 = x_2(x_3-x_1)&\\
        &\dot{x}_3 = x_3(x_1-x_2)&\\[.25em]
        &x_1+x_2+x_3 = C_1&\\
        &x_1x_2x_3 = C_2&
    \end{aligned}
\end{equation*}
If we consider only classes of functions in the library, the algorithm will gravitate towards the library containing only first order polynomials. If subsets of $\Theta = [\textbf{x}, \textbf{x}^2,\textbf{x}^3]$ are now included, particularly the subsets of $[\textbf{x},\textbf{x}^3]$, we are more likely to recover both laws. We do note that for $\Theta = [x_1, x_2,x_3,x_1x_2x_3]$, we recover both laws exactly, however the corresponding $\delta$ is within machine epsilon of the $\delta$ for $\Theta = [x_1,x_2,x_3]$, indicating that it may be advantageous to investigate several libraries with the largest $\delta$ to ensure all possible conservation laws are recovered.

Building off these observations, constructing a global library of all possible functions and iterating over the power set of this global library, excluding the empty set and sets containing a single state variable, seems to be the next logical step. While this approach will more accurately infer all possible conservation laws, it will be significantly more computationally demanding for if $\Theta_{\textrm{global}}^T\in \mathbb{R}^{k}$, $2^{k}-1-k$ sub-libraries should be considered in the algorithm.
Moreover, we observed for systems with multiple conservation laws originating from different sub-libraries, $\delta$ can be significantly close. To alleviate this, we propose user supervision might be needed if there is a close decision. 

Future work will be focused on ways to derive minimal consistent symbolic representations of the recovered conserved quantities and applying the methodology to experimental data.

\section{Acknowledgements}
The authors gratefully acknowledge fruitful discussions with Mariaelena Pieorbon. ME was supported in part by the Simons Foundation grant \#854541. TO was partially funded by the George Mason University Provost through their Summer Research Fellowship and Dissertation Completion Grant.

    \bibliographystyle{ieeetr}

  \bibliography{references}

  \newpage
  \appendix
\section{Supplemental Notes}
Table \ref{tab:theta_lib} contains a list of all $\Theta$ libraries used for the analysis in Section \ref{sec:testing}. 
\begingroup
\begin{table}[H]
    \caption{Table of $\Theta$ libraries considered for all graphs in Section \ref{sec:testing}.}
    \centering
    \renewcommand{\arraystretch}{1.5}
    \begin{tabular}{c|l}
         Triple&  Library\\ \hline
         (1, 0, 0)& $\Theta = [\textbf{X}]$ \\ \hline
    (2, 0, 0)& $\Theta = [\textbf{X}, \textbf{X}^2]$ \\ \hline
    (3, 0, 0)& $\Theta = [\textbf{X}, \textbf{X}^2, \textbf{X}^3]$ \\ \hline
    (1, 1, 0)& $\Theta = [\textbf{X}, \sin(\textbf{X}), \cos(\textbf{X})]$ \\ \hline
    (2, 1, 0)& $\Theta = [\textbf{X}, \textbf{X}^2, \sin(\textbf{X}), \cos(\textbf{X})]$ \\ \hline
    (3, 1, 0)& $\Theta = [\textbf{X}, \textbf{X}^2, \textbf{X}^3 \sin(\textbf{X}), \cos(\textbf{X})]$ \\ \hline
    (1, 0, 1)& $\Theta = [\textbf{X}, \ln(\textbf{X})]$ \\ \hline
    (2, 0, 1)& $\Theta = [\textbf{X}, \textbf{X}^2, \ln(\textbf{X})]$ \\ \hline
    (3, 0, 1)& $\Theta = [\textbf{X}, \textbf{X}^2, \textbf{X}^3, \ln(\textbf{X})]$ \\ \hline
    (1, 1, 1)& $\Theta = [\textbf{X}, \sin(\textbf{X}), \cos(\textbf{X}), \ln(\textbf{X})]$ \\ \hline
    (2, 1, 1)& $\Theta = [\textbf{X}, \textbf{X}^2, \sin(\textbf{X}), \cos(\textbf{X}), \ln(\textbf{X})]$ \\ \hline
    (3, 1, 1)& $\Theta = [\textbf{X}, \textbf{X}^2, \textbf{X}^3 \sin(\textbf{X}), \cos(\textbf{X}), \ln(\textbf{X})]$   
    \end{tabular}
    \label{tab:theta_lib}
\end{table}
\endgroup

The following results correspond to the figures presented in Section \ref{sec:testing}. Table \ref{tab:reference} details what values will be recorded. 

 \begingroup
\begin{table}[H]
\caption{Reference table}
 \renewcommand{\arraystretch}{1.5}
\centering
\begin{tabular}{c|l}
Label & Interpretation \\\hline
$\varepsilon_\textbf{x}$ & Error in the data - $\varepsilon_{\textbf{x}} = \lVert \textbf{x} - \tilde{\textbf{x}}\rVert $ \\
Model&   ordered triple corresponding to the form of the $\Theta$-library. See Table \ref{tab:theta_lib} for reference. \\
 len($\Theta$) & Length of $\Theta$\\ 
$\sigma_{\textrm{cutoff}}$ & Cutoff value used for library choice\\
$\delta$ & $\sigma$- difference as calculated in Equation \eqref{eq:sigmagap} \\ 
count &count as calculated in Equation \eqref{eq:count} taking into consideration the threshold value \\ 
$\sum_j\lVert \Gamma \xi_{cons,j}\rVert_2$ & Sum of 2-norms of found conservation(s) 
\end{tabular}
\label{tab:reference}
\end{table}
\endgroup
\newpage
\subsection{Example 1}
Below are the tables corresponding to the figures in Section \ref{sec:examples-volpert}. 

\begingroup
\begin{table}[H]
    \centering
    \renewcommand{\arraystretch}{1.1}
    \small{
    \begin{tabular}{|c|c|c|c|c|c|c|}\hline
     $\varepsilon_\textbf{x}$ & Model & $\textrm{len}(\Theta)$ & $\sigma_{\textrm{cutoff}}$ & $\delta$& count & $\sum_j\lVert \Gamma*\xi_j\rVert$ \\ \hlinewd{1.5pt}
                        \multirow{9}*{0}
                                & \textbf{(1, 0, 0)} & \textbf{3.0}& \textbf{1.0e-10} & \textbf{4.0579} & \textbf{1.0} & \textbf{2.1789e-14}\\\cline{2-7}
                                & (2, 0, 0) & 9.0 & 1.0e-10 & 0.23527 & 4.0 & 8.1832e-14\\ \cline{2-7}
                                & (3, 0, 0) & 19.0 & 1.0e-10 & 0.0039731 & 10.0 & 1.7251e-13\\ \cline{2-7}
                                & (1, 0, 1)  & 6.0 & 1.0e-10  & 0.084417 & 1.0 &8.3465e-14\\ \cline{2-7}
                                & (2, 0, 1) & 12.0 & 1.0e-10& 0.0021132 & 4.0 & 3.5442e-12\\ \cline{2-7}
                                & (1, 1, 0) &9.0 & 1.0e-10 & 0.0039081 & 1.0 & 4.2259e-12\\\cline{2-7}
                                & (2, 1, 0)  & 15.0 & 1.0e-10  & 0.000077496 & 4.0 &9.0724e-12\\ \cline{2-7}
                                & (1, 1, 1) & 12.0 &1.0e-10 & 0.000023991 & 1.0 & 1.4756e-14\\ \cline{2-7}
                                & (2, 1, 1)  & 18.0 & 1.0e-10 & 3.647e-7 & 4.0 & 4.4435e-14\\ \hline
                                \multicolumn{7}{|c|}{\textbf{Recovered Law:} $0.57735x_1+0.57735x_2+0.57735x_3 = C$}\\ 
                                \multicolumn{7}{|c|}{\textbf{Reduced Law:} $x_1+x_2+x_3 = C$} \\ \hhline{|=|=|=|=|=|=|=|}
                        \multirow{9}*{5.4849e-10}
                                & \textbf{(1, 0, 0)}& \textbf{3.0} & \textbf{5.1902e-6} & \textbf{4.0579} & \textbf{1.0} & \textbf{3.3263e-8}\\ \cline{2-7}
                                & (2, 0, 0)& 9.0 & 8.9898e-6  & 0.23527 & 4.0 & 1.1613e-7\\ \cline{2-7}
                                & (3, 0, 0)& 19.0 & 0.000013062 & 0.0039729 & 10.0 & 2.2526e-7\\ \cline{2-7}
                                & (1, 0, 1)& 6.0 & 8.9898e-6 & 0.0039081 & 1.0 & 3.0055e-8\\ \cline{2-7}
                                & (2, 0, 1)& 12.0 & 0.00001038  & 0.0021131 & 4.0 & 1.0381e-7\\\cline{2-7} 
                                & (1, 1 , 0)& 9.0 & 8.9898e-6 & 0.0039081 & 1.0 & 3.0055e-8\\ \cline{2-7}
                                & (2, 1 , 0)& 15.0 & 0.000011606 & 0.00007741 & 4.0 & 9.1386e-8\\\cline{2-7}
                                & (1, 1 , 1)& 12.0 & 0.00001038 &  0.000023963 & 1.0 & 2.6291e-8\\ \cline{2-7}
                                & (2, 1 , 1)& 18.0 & 0.000012713 &  0.000014258 & 6.0 & 8.735e-6\\ \hline
                                \multicolumn{7}{|c|}{\textbf{Recovered Law:} $-0.57735x_1-0.57735x_2-0.57735x_3 = C$}\\ 
                                \multicolumn{7}{|c|}{\textbf{Reduced Law:} $x_1+x_2+x_3 = C$} \\ \hhline{|=|=|=|=|=|=|=|}
                         \multirow{9}*{4.3606e-05}
                                & \textbf{(1, 0, 0)}& \textbf{3.0}& \textbf{0.0095964} & \textbf{4.0563} & \textbf{1.0} &\textbf{0.001608}\\ \cline{2-7}
                                 & (2, 0, 0)& 9.0 & 0.016621 & 0.22969 & 4.0 & 0.00534\\ \cline{2-7}
                                 & (3, 0, 0)& 19.0 & 0.02415 & 0.035632 & 11.0 & 0.015822\\ \cline{2-7}
                                 & (1, 0, 1)& 6.0 & 0.013571  & 0.082811 & 1.0 &0.0015219\\ \cline{2-7}
                                 & (2, 0, 1)& 12.0 & 0.019193  & 0.01559 & 5.0 & 0.0071915\\ \cline{2-7}
                                 & (1, 1 , 0)& 9.0 &0.016621 & 0.016893& 2.0 & 0.005356\\ \cline{2-7}
                                 & (2, 1 , 0)& 15.0 & 0.021458  & 0.027648 & 8.0 & 0.012539\\ \cline{2-7}
                                 & (1, 1 , 1)& 12.0 & 0.019193 & 0.023277 & 5.0 & 0.0093485\\ \cline{2-7}
                                 & (2, 1 , 1)& 18.0 & 0.023506  & 0.037703 & 11.0 &0.017979\\\hline

                                 \multicolumn{7}{|c|}{\textbf{Recovered Law:} $-0.57732x_1-0.57739x_2-0.57735x_3 = C$}\\ 
                                \multicolumn{7}{|c|}{\textbf{Reduced Law:} $0.9999x_1+x_2+0.9999x_3 = C$} \\  \hline
\end{tabular}
}%
    \caption{Table of values corresponding to $N=20$ for Example 1 (the left graphs of Figure \ref{fig:volpert_sing})}
    \label{tab:volpert_sing20}
\end{table}
\endgroup

\begingroup
\begin{table}[H]
    \centering
    \renewcommand{\arraystretch}{1.1}
    \small{
    \begin{tabular}{|c|c|c|c|c|c|c|}\hline
     $\varepsilon_\textbf{x}$ & Model & $\textrm{len}(\Theta)$ & $\sigma_{\textrm{cutoff}}$ & $\delta$& count & $\sum_j\lVert \Gamma*\xi_j\rVert$ \\ \hlinewd{1.5pt}
                        \multirow{12}*{0.0}                                
                                & \textbf{(1, 0, 0)}& \textbf{3.0} & \textbf{1.0e-10}  &\textbf{ 8.8148} & \textbf{1.0} & \textbf{3.5381e-14}\\\cline{2-7}
                                & (2, 0, 0)& 9.0 & 1.0e-10  & 0.37496 & 4.0 & 1.4089e-13\\ \cline{2-7}
                                & (3, 0, 0)& 19.0 &1.0e-10  & 0.0032626 & 10.0 & 2.9205e-13\\ \cline{2-7}
                                & (1, 0, 1)& 6.0 & 1.0e-10 & 0.023855 & 1.0 &4.6508e-14\\ \cline{2-7}
                                & (2, 0, 1)& 12.0 &1.0e-10  & 0.0033704 & 4.0 & 1.5425e-12\\ \cline{2-7}
                                & (3, 0, 1)& 22.0 & 1.0e-10  & 0.00006919 & 11.0 & 2.9114e-13\\ \cline{2-7}
                                & (1, 1, 0)& 9.0 & 1.0e-10  &0.0063047 & 1.0 & 8.028e-12\\\cline{2-7}
                                & (2, 1, 0)& 15.0 & 1.0e-10  & 0.000032995 & 4.0 & 1.1258e-13\\ \cline{2-7} 
                                & (3, 1, 0)& 25.0 & 1.0e-10  & 6.6346e-7 & 10.0 & 2.6933e-13\\ \cline{2-7}
                                & (1, 1, 1)& 12.0 &1.0e-10  & 5.1936e-6 & 1.0 & 2.3335e-14\\ \cline{2-7}
                                & (2, 1, 1)& 18.0 & 1.0e-10  & 5.581e-7 & 4.0 & 1.0469e-13\\\cline{2-7}
                                & (3, 1, 1)& 28.0 & 1.0e-10  & 4.0946e-8 & 11.0 & 3.1179e-13 \\ \hline
                                
                                \multicolumn{7}{|c|}{\textbf{Recovered Law:} $0.57735x_1+0.57735x_2+0.57735x_3 = C$}\\ 
                                \multicolumn{7}{|c|}{\textbf{Reduced Law:} $x_1+x_2+x_3 = C$} \\ \hhline{|=|=|=|=|=|=|=|}   
                        \multirow{12}*{4.7594e-09}
                                & \textbf{(1, 0, 0)}& \textbf{3.0} & \textbf{0.000010558} &\textbf{ 8.8148} & \textbf{1.0} &\textbf{2.0206e-8}\\ \cline{2-7}
                                & (2, 0, 0)& 9.0 & 0.000018288 & 0.37496 & 4.0 & 6.7902e-8\\ \cline{2-7}
                                & (3, 0, 0)& 19.0 & 0.000026571  & 0.0032625 & 10.0 & 1.5507e-7\\ \cline{2-7}
                                & (1, 0, 1)& 6.0 & 0.000014932  &0.023855 & 1.0 & 1.9563e-8\\ \cline{2-7}
                                & (2, 0, 1)& 12.0 & 0.000021117  & 0.0033704 & 4.0 &6.5934e-8\\ \cline{2-7}
                                & (3, 0, 1)& 22.0 & 0.000028592  & 0.000069042 & 11.0 & 1.5934e-7\\ \cline{2-7}
                                & (1, 1, 0)& 9.0 & 0.000018288  & 0.0063047 & 1.0 & 1.8616e-8\\ \cline{2-7}
                                & (2, 1, 0)& 15.0 & 0.000023609 & 0.000032933 & 4.0 & 6.5453e-8\\ \cline{2-7}
                                & (3, 1, 0)& 25.0 & 0.000030479  & 0.000083452 &13.0 & 6.8034e-6\\ \cline{2-7}
                                & (1, 1, 1)& 12.0 & 0.000021117 & 0.00026537 & 2.0 & 5.2107e-6\\\cline{2-7}
                                & (2, 1, 1)& 18.0 & 0.000025863 & 0.000057904 & 6.0 & 2.2486e-6\\ \cline{2-7}
                                & (3, 1, 1)& 28.0 & 0.000032256 & 0.000081564 & 16.0 & 0.000020748 \\ \hline
                                \multicolumn{7}{|c|}{\textbf{Recovered Law:} $-0.57735x_1-0.57735x_2-0.57735x_3 = C$}\\ 
                                \multicolumn{7}{|c|}{\textbf{Reduced Law:} $x_1+x_2+x_3 = C$} \\ \hhline{|=|=|=|=|=|=|=|} 
                         \multirow{12}*{4.5790e-05}
                                & \textbf{(1, 0, 0)}& \textbf{3.0} & \textbf{0.022169} &  \textbf{8.8123} & \textbf{1.0} &\textbf{0.0027741}\\ \cline{2-7}
                                & (2, 0, 0)& 9.0 &0.038397 & 0.36689 & 4.0 & 0.0094069\\ \cline{2-7}
                                & (3, 0, 0)& 19.0 & 0.05579 &  0.059744 & 12.0 & 0.076484\\ \cline{2-7}
                                & (1, 0, 1)& 6.0 & 0.031351 & 0.503 &2.0 & 0.026322\\ \cline{2-7}
                                & (2, 0, 1)& 12.0 & 0.044337 & 0.37732 & 7.0 & 0.077914 \\ \cline{2-7}
                                & (3, 0, 1)& 22.0 &0.060033  & 0.056876 & 15.0 & 0.08438\\ \cline{2-7}
                                & (1, 1, 0)& 9.0 &0.038397 &0.03069 & 3.0 & 0.020124\\ \cline{2-7}
                                & (2, 1, 0)& 15.0 & 0.049571  & 0.063248 & 9.0 & 0.041129\\ \cline{2-7}
                                & (3, 1, 0)& 25.0 & 0.063995 & 0.061198 & 18.0 & 0.081371\\ \cline{2-7}
                                & (1, 1, 1)& 12.0 & 0.044337 & 0.032679 & 6.0 & 0.043595\\ \cline{2-7}
                                & (2, 1, 1)& 18.0 & 0.054302  & 0.062429 & 12.0& 0.06318\\ \cline{2-7}
                                & (3, 1, 1)& 28.0 &  0.067726 & 0.058391 & 21.0 & 0.088529 \\ \hline
                                \multicolumn{7}{|c|}{\textbf{Recovered Law:} $0.57731x_1+0.57737x_2+0.57737x_3 = C$}\\ 
                                \multicolumn{7}{|c|}{\textbf{Reduced Law:} $0.9999x_1+x_2+0.9999x_3 = C$} \\\hline
\end{tabular}
}%
    \caption{Table of values corresponding to $N=100$ for Example 1 (the right graphs of Figure \ref{fig:volpert_sing})}
    \label{tab:volpert_sing100}
\end{table}
\endgroup

\subsection{Example 2}
Below are the tables corresponding to the figures in Section \ref{sec:examples-twocons}. 

\begingroup
\begin{table}[H]
    \centering
    \renewcommand{\arraystretch}{1.1}
    \small{
    \begin{tabular}{|c|c|c|c|c|c|c|}\hline
     $\varepsilon_\textbf{x}$ & Model & $\textrm{len}(\Theta)$ & $\sigma_{\textrm{cutoff}}$ & $\delta$& count & $\sum_j\lVert \Gamma*\xi_j\rVert$ \\ \hlinewd{1.5pt}
                        \multirow{6}*{0}
                                & \textbf{(1, 0, 0)} & \textbf{4.0} & \textbf{1.0e-10 }& \textbf{1.8474} & \textbf{2.0} & \textbf{1.4716e-14}\\ \cline{2-7}
                                & (2, 0, 0) & 14.0 & 1.0e-10 &  0.063937 & 9.0 & 4.1111e-14\\ \cline{2-7}
                                & (1, 0, 1) & 8.0 &  1.0e-10 &  0.0034213 & 2.0 &1.3616e-12\\ \cline{2-7}
                                & (2, 0, 1) & 18.0 &  1.0e-10 &  0.00013746 & 9.0 & 6.3428e-12\\ \cline{2-7}
                                & (1, 1, 0) & 12.0 & 1.0e-10 & 0.00030492 & 4.0 & 1.7433e-12\\ \cline{2-7}
                                & (1, 1, 1) & 16.0 &1.0e-10 & 3.8709e-7 & 4.0 & 1.0844e-11 \\ \hline

                                \multicolumn{7}{|c|}{\textbf{Recovered Law:} {$\!%
                                \begin{aligned} &-0.31642x_1+0.63246x_2+0.31604x_3+0.63246x_4 = C_1\\
                                & 0.70702x_1+0.00017x_2+0.70719x_3+0.00017x_4 = C_2
                                \end{aligned}
                                                        $} }\\ 
                                                        \multicolumn{7}{|c|}{}\\
                                \multicolumn{7}{|c|}{\textbf{Reduced Law:} {$\!%
                                                         \begin{aligned} &x_1+x_3 = C_1\\
                                                         &x_2+x_3+x_4=C_2
                                                         \end{aligned}
                                                        $} } \\ \hhline{|=|=|=|=|=|=|=|}
                        \multirow{6}*{5.4588e-10}
                                & \textbf{(1, 0, 0)} & \textbf{4.0} & \textbf{5.9742e-6}  & \textbf{1.8474} & \textbf{2.0} & \textbf{3.4017e-8}\\ \cline{2-7}
                                & (2, 0, 0) & 14.0 & 0.000011177  & 0.063937 & 9.0 & 9.1189e-8\\ \cline{2-7}
                                & (1, 0, 1) & 8.0 & 8.4488e-6 & 0.0034213 & 2.0 & 3.0698e-8\\ \cline{2-7}
                                & (2, 0, 1) & 18.0 &0.000012673 &  0.00013741 & 9.0 &8.0334e-8\\ \cline{2-7}
                                & (1, 1, 0) & 12.0 & 0.000010348  & 0.0003049 & 4.0 & 3.4966e-8\\ \cline{2-7}
                                & (1, 1, 1) & 16.0 & 0.000011948  & 0.000017266 & 5.0& 4.149e-7\\\hline

                                \multicolumn{7}{|c|}{\textbf{Recovered Law:} {$\!%
                                \begin{aligned} &0.47612x_1-0.61411x_2-0.13799x_3-0.61411x_4 = C_1\\
                                & -0.61099x_1-0.15122x_2-0.76221x_3-0.15122x_4 = C_2
                                \end{aligned}
                                                        $} }\\ 
                                                        \multicolumn{7}{|c|}{}\\
                                \multicolumn{7}{|c|}{\textbf{Reduced Law:} {$\!%
                                                         \begin{aligned} &x_1+x_3= C_1\\
                                                         &x_2+x_3+x_4=C_2
                                                         \end{aligned}
                                                        $} } \\ \hhline{|=|=|=|=|=|=|=|}
                         \multirow{6}*{5.5253e-05}
                                & \textbf{(1, 0, 0) }& \textbf{4.0} & \textbf{0.012975 }&  \textbf{1.8447} & \textbf{2.0} &\textbf{0.0041992}\\ \cline{2-7}
                                & (2, 0, 0) & 14.0 & 0.024275 & 0.056546 & 9.0 & 0.011785\\ \cline{2-7}
                                & (1, 0, 1) & 8.0 & 0.01835 & 0.11553 &4.0 & 0.025573\\ \cline{2-7}
                                & (2, 0, 1) & 18.0 & 0.027525  & 0.10737 & 13.0 & 0.037722\\ \cline{2-7}
                                & (1, 1, 0) & 12.0 & 0.022474  &0.020826 & 7.0 & 0.014176\\ \cline{2-7}
                                & (1, 1, 1) & 16.0 & 0.025951 & 0.070681 & 11.0 & 0.024704\\\hline
                                \multicolumn{7}{|c|}{\textbf{Recovered Law:} {$\!%
                                \begin{aligned} &-0.05178x_1+0.5933x_2+0.5414x_3+0.59343x_4 = C_1\\
                                & -0.7728x_1+0.2189x_2-0.5538x_3+0.2189x_4 = C_2
                                \end{aligned}
                                                        $} }\\ 
                                                        \multicolumn{7}{|c|}{}\\
                                \multicolumn{7}{|c|}{\textbf{Reduced Law:} {$\!%
                                                         \begin{aligned} &x_1+0.9998x_3-(3.7e-5)x_2 = C_1\\
                                                         &0.9998x_2+0.9996x_3+x_4=C_2
                                                         \end{aligned}
                                                        $} } \\ \hline
\end{tabular}
}%
    \caption{Table of values corresponding to $N=20$ for Example 2 (the left graphs of Figure \ref{fig:twocons_sing})}
    \label{tab:twocons_sing20}
\end{table}
\endgroup

\begingroup
\begin{table}[H]
    \centering
    \renewcommand{\arraystretch}{1.1}
    \footnotesize{
    \begin{tabular}{|c|c|c|c|c|c|c|}\hline
     $\varepsilon_\textbf{x}$ & Model & $\textrm{len}(\Theta)$ & $\sigma_{\textrm{cutoff}}$ & $\delta$& count & $\sum_j\lVert \Gamma*\xi_j\rVert$ \\ \hlinewd{1.5pt}
                        \multirow{12}*{0.0}
                                & \textbf{(1, 0, 0)} & \textbf{4.0 }& \textbf{1.0e-10} & \textbf{3.9512} & \textbf{2.0} & \textbf{3.0341e-14}\\ \cline{2-7}
                                & (2, 0, 0) & 14.0 & 1.0e-10 & 0.025957 & 9.0 & 9.1177e-14 \\ \cline{2-7}
                                & (3, 0, 0) & 34.0 &1.0e-10 & 0.00029318 & 25.0 & 2.3278e-13\\ \cline{2-7}
                                & (1, 0, 1) & 8.0 & 1.0e-10 & 0.0036841 & 2.0 &3.9258e-12\\ \cline{2-7}
                                & (2, 0, 1) & 18.0 &1.0e-10 &  0.00013799 & 9.0 & 8.1403e-12\\ \cline{2-7}
                                & (3, 0, 1) & 38.0 & 1.0e-10 & 3.6522e-6 & 25.0 & 8.3401e-11\\ \cline{2-7}
                                & (1, 1, 0) & 12.0 & 1.0e-10 & 0.000094638 & 4.0 & 5.8228e-12\\ \cline{2-7}
                                & (2, 1, 0) & 22.0 & 1.0e-10 & 2.5575e-6 & 11.0 &1.2546e-11\\ \cline{2-7}
                                & (3, 1, 0) & 42.0 & 1.0e-10 &  2.1804e-8 & 27.0 & 2.0866e-13\\ \cline{2-7}
                                 & (1, 1, 1) & 16.0 & 1.0e-10 &  1.3843e-6 & 4.0 & 2.0776e-11\\ \cline{2-7}
                                & (2, 1, 1) & 26.0 &1.0e-10 &  1.8045e-8 & 11.0 & 2.1566e-11\\ \cline{2-7}
                                & (3, 1, 1) & 46.0 &1.0e-10 & 1.5864e-9 & 28.0 &8.3203e-11 \\ \hline
                                \multicolumn{7}{|c|}{\textbf{Recovered Law:} {$\!%
                                \begin{aligned} &0.3169x_1-0.6324x_2-0.3154x_3-0.6324x_4 = C_1\\
                                & -0.7067x_1-0.0006x_2-0.7074x_3-0.0006x_4 = C_2
                                \end{aligned}
                                                        $} }\\ 
                                                        \multicolumn{7}{|c|}{}\\
                                \multicolumn{7}{|c|}{\textbf{Reduced Law:} {$\!%
                                                         \begin{aligned} &x_1+0.9999x_3 = C_1\\
                                                         &x_2+0.9999x_3+x_4=C_2
                                                         \end{aligned}
                                                        $} } \\  \hhline{|=|=|=|=|=|=|=|}   
                        \multirow{12}*{6.5220e-10}
                                & \textbf{(1, 0, 0)} & \textbf{4.0} & \textbf{0.000015041}  & \textbf{3.9512} & \textbf{2.0} &\textbf{4.7404e-8}\\ \cline{2-7}
                                & (2, 0, 0) & 14.0 & 0.00002814 & 0.025957 & 9.0 & 1.336e-7\\ \cline{2-7}
                                & (3, 0, 0) & 34.0 &0.000043852 & 0.00029305 & 25.0 & 2.5312e-7\\ \cline{2-7}
                                & (1, 0, 1) & 8.0 & 0.000021272 &0.0036841 & 2.0 & 4.4098e-8\\ \cline{2-7}
                                & (2, 0, 1) & 18.0 &  0.000031907  & 0.00013793 & 9.0 &1.1969e-7\\ \cline{2-7}
                                & (3, 0, 1) & 38.0 &0.00004636 &  0.000074965 & 27.0 & 0.000010574 \\ \cline{2-7}
                                & (1, 1, 0) & 12.0 &0.000026052 & 0.000094605 & 4.0 & 5.228e-8 \\ \cline{2-7}
                                & (2, 1, 0) & 22.0 & 0.000035275& 0.00007154 & 13.0 & 0.000015945\\ \cline{2-7}
                                & (3, 1, 0) & 42.0 &0.000048739 &0.00018352 & 32.0 & 7.4805e-6 \\ \cline{2-7}
                                 & (1, 1, 1) & 16.0 &0.000030083  & 0.000077422 & 6.0& 0.000011873 \\ \cline{2-7}
                                & (2, 1, 1) & 26.0 & 0.000038348 & 0.000074075 & 16.0 & 0.000021881\\ \cline{2-7}
                                & (3, 1, 1) & 46.0 & 0.000051007 & 0.000072377 & 35.0 & 0.000015447 \\ \hline
                                \multicolumn{7}{|c|}{\textbf{Recovered Law:} {$\!%
                                \begin{aligned} &0.6582x_1-0.5237x_2+0.1345x_3-0.5237x_4 = C_1\\
                                & 0.4082x_1+0.3545x_2+0.7628x_3+0.3545x_4 = C_2
                                \end{aligned}
                                                        $} }\\ 
                                                        \multicolumn{7}{|c|}{}\\
                                \multicolumn{7}{|c|}{\textbf{Reduced Law:} {$\!%
                                                         \begin{aligned} &x_1+x_3x_4 = C_1\\
                                                         &x_2+x_3+x_4=C_2
                                                         \end{aligned}
                                                        $} } \\  \hhline{|=|=|=|=|=|=|=|} 
                         \multirow{12}*{6.6525e-05}
                                & \textbf{(1, 0, 0)} & \textbf{4.0} & \textbf{0.032836}  & \textbf{3.9484} & \textbf{2.0} &\textbf{0.0047325}\\ \cline{2-7}
                                & (2, 0, 0) & 14.0 & 0.061431 & 0.26587 & 10.0 & 0.0397\\ \cline{2-7}
                                & (3, 0, 0) & 34.0 &0.095733 & 0.096853 & 29.0 & 0.12524\\ \cline{2-7}
                                & (1, 0, 1) & 8.0 & 0.046437 &  0.17064 &4.0 & 0.029846\\ \cline{2-7}
                                & (2, 0, 1) & 18.0 & 0.069656 & 0.32509 & 14.0 & 0.11871\\ \cline{2-7}
                                & (3, 0, 1) & 38.0 & 0.10121 & 0.10387 & 33.0 & 0.1401\\ \cline{2-7}
                                & (1, 1, 0) & 12.0 & 0.056874  &0.13477 & 8.0 & 0.035431\\ \cline{2-7}
                                & (2, 1, 0) & 22.0 & 0.077008  & 0.30556 & 18.0 & 0.068139\\ \cline{2-7}
                                & (3, 1, 0) & 42.0 & 0.1064 & 0.09839 & 37.0 & 0.12802\\ \cline{2-7}
                                 & (1, 1, 1) & 16.0 & 0.065672  & 0.22147 & 12.0 & 0.088162\\ \cline{2-7}
                                & (2, 1, 1) & 26.0 &0.083716  & 0.35581 & 22.0 &0.1323 \\ \cline{2-7}
                                & (3, 1, 1) & 46.0 &0.11135  & 0.10566 & 41.0 & 0.14254 \\ \hline

                                \multicolumn{7}{|c|}{\textbf{Recovered Law:} {$\!%
                                \begin{aligned} &-0.1717x_1-0.5057x_2-0.6775x_3-0.5056x_4 = C_1\\
                                & 0.7552x_1-0.3798x_2+0.3755x_3-0.3798x_4 = C_2
                                \end{aligned}
                                                        $} }\\ 
                                                        \multicolumn{7}{|c|}{}\\
                                \multicolumn{7}{|c|}{\textbf{Reduced Law:} {$\!%
                                                         \begin{aligned} &x_1+x_3-(5.3e-5)x_4 = C_1\\
                                                         &x_2+1.0002x_3+0.9999x_4=C_2
                                                         \end{aligned}
                                                        $} } \\ \hline
\end{tabular}
}%
    \caption{Table of values corresponding to $N=100$ for Example 2 (the right graphs of Figure \ref{fig:twocons_sing})}
    \label{tab:twocons_sing100}
\end{table}
\endgroup

\subsection{Example 3}
Below are the tables corresponding to the figures in Section \ref{sec:examples-chemcial}.

\begingroup
\begin{table}[H]
    \centering
    \renewcommand{\arraystretch}{1.1}
    \small{
    \begin{tabular}{|c|c|c|c|c|c|c|}\hline
     $\varepsilon_\textbf{x}$ & Model & $\textrm{len}(\Theta)$ & $\sigma_{\textrm{cutoff}}$ & $\delta$& count & $\sum_j\lVert \Gamma*\xi_j\rVert$ \\ \hlinewd{1.5pt}
                        \multirow{9}*{0}
                               & (1, 0, 0) & 3.0 & 0.00001  & NaN & 0 & NaN \\ \cline{2-7}
                                & (2, 0, 0) & 9.0 &  0.00001 & 0.000072085 & 3.0 & 0.000011441\\ \cline{2-7}
                                & (3, 0, 0) & 19.0 & 0.00001 & 0.000045224 & 11.0 & 4.4834e-6 \\ \cline{2-7}
                                & \textbf{(1, 0, 1)} & \textbf{6.0} & \textbf{0.00001 } &\textbf{0.0027865} & \textbf{1.0} & \textbf{1.9743e-6 }\\ \cline{2-7}
                                & (2, 0, 1) & 12.0 & 0.00001 & 0.000027509 & 4.0 & 1.9893e-6 \\ \cline{2-7}
                                & (1, 1, 0) & 9.0 & 0.00001  & 0.000017707 & 2.0 & 6.1443e-6 \\ \cline{2-7}
                                & (2, 1, 0) & 15.0 & 0.00001 & 8.0627e-6 & 7.0 & 3.5526e-6 \\ \cline{2-7}
                                & (1, 1, 1) & 12.0 & 0.00001  & 0.000014728 & 4.0 & 9.7154e-6\\ \cline{2-7}
                                & (2, 1, 1) & 18.0 & 0.00001  & 0.000010709 & 9.0 & 1.3221e-6\\ \hline
                                \multicolumn{7}{|c|}{\textbf{Recovered Law:} {$\!%
                                \begin{aligned} 
                                -0.34852x_1&-0.34831x_2+0.0003x_3+0.0001\ln(x_1)\\
                                &-0.00005\ln(x_2)+0.87018\ln(x_3) = C \end{aligned}
                                                        $} }\\ 
                                                        \multicolumn{7}{|c|}{}\\
                                \multicolumn{7}{|c|}{\textbf{Reduced Law:} {$\!%
                                \begin{aligned} 
                                x_1&+0.9994x_2-0.0008x_3-0.0003\ln(x_1)\\
                                &+0.0001\ln(x_2)-2.4968\ln(x_3) = C
                                \end{aligned}
                                                        $} } \\ \hhline{|=|=|=|=|=|=|=|}
                        \multirow{9}*{4.5634e-10}
                                & (1, 0, 0) & 3.0 & 0.00001  & NaN & 0 & NaN  \\ \cline{2-7}
                                & (2, 0, 0) & 9.0 &  0.00001 & 0.000072089 & 3.0 & 0.000011438 \\ \cline{2-7}
                                & (3, 0, 0) & 19.0 & 0.000011554  & 0.000045221 & 11.0 & 4.4825e-6 \\ \cline{2-7}
                                & \textbf{(1, 0, 1)} & \textbf{6.0} & \textbf{0.00001 } &\textbf{0.0027865 }& \textbf{1.0} &\textbf{ 1.9731e-6} \\ \cline{2-7}
                                & (2, 0, 1) & 12.0 & 0.00001  & 0.000027509 & 4.0 & 1.989e-6  \\ \cline{2-7}
                                & (1, 1, 0) & 9.0 &  0.00001  & 0.000017708 & 2.0 & 6.1442e-6\\ \cline{2-7}
                                & (2, 1, 0) & 15.0 & 0.000010266  &8.0607e-6 & 7.0 & 3.5526e-6 \\ \cline{2-7}
                                & (1, 1, 1) & 12.0 & 0.00001 & 0.000014728 & 4.0 & 9.7154e-6\\ \cline{2-7}
                                & (2, 1, 1) & 18.0 & 0.000011246  & 0.000010709 & 9.0 & 1.3218e-6\\ \hline
                                \multicolumn{7}{|c|}{\textbf{Recovered Law:} {$\!%
                                \begin{aligned} 
                                -0.34852x_1&-0.34831x_2+0.0003x_3+0.0001\ln(x_1)\\
                                &-0.00005\ln(x_2)+0.87018\ln(x_3) = C\end{aligned}
                                                        $}}\\
                                                         \multicolumn{7}{|c|}{}\\
                                \multicolumn{7}{|c|}{\textbf{Reduced Law:} {$\!%
                                \begin{aligned} x_1&+0.9994x_2-0.0008x_3-0.0001\ln(x_1)\\
                                &+0.0001\ln(x_2)-2.4968\ln(x_3) = C\end{aligned}
                                                        $}} \\ \hhline{|=|=|=|=|=|=|=|}
                         \multirow{9}*{3.3100e-05}
                                & (1, 0, 0) & 3.0 & 0.0079854  & NaN & 0 &NaN \\ \cline{2-7}
                                & (2, 0, 0) & 9.0 &  0.013831 & 0.014003 & 4.0 & 0.0044729\\ \cline{2-7}
                                & (3, 0, 0) & 19.0 & 0.020096 &  0.042745 & 14.0 & 0.014357 \\ \cline{2-7}
                                & \textbf{(1, 0, 1)} & \textbf{6.0} & \textbf{0.011293} & \textbf{0.020824}& \textbf{1.0} & \textbf{0.00095611} \\ \cline{2-7}
                                & (2, 0, 1) & 12.0 &  0.015971 & 0.058471 & 7.0 & 0.0069299\\ \cline{2-7}
                                & (1, 1, 0) & 9.0 & 0.013831  &0.021211 & 5.0 & 0.015767 \\ \cline{2-7}
                                & (2, 1, 0) & 15.0 & 0.017856 & 0.019962 & 10.0 & 0.0051929 \\ \cline{2-7}
                                & (1, 1, 1) & 12.0 & 0.015971  & 0.035169 & 7.0 & 0.0047944\\ \cline{2-7}
                                & (2, 1, 1) & 18.0 & 0.01956  &0.059301 & 13.0 & 0.0085682\\ \hline
                                \multicolumn{7}{|c|}{\textbf{Recovered Law:} {$\!%
                                \begin{aligned} -0.0378x_1&-0.0630x_2-0.0352x_3-0.0367\ln(x_1)\\
                                &-0.0228\ln(x_2)+0.93182\ln(x_3) = C\end{aligned}
                                                        $} }\\
                                                         \multicolumn{7}{|c|}{}\\
                                \multicolumn{7}{|c|}{\textbf{Reduced Law:} {$\!%
                                \begin{aligned} 
                                x_1&+1.665x_2+9.3197x_3+0.9798\ln(x_1)+0.6024\ln(x_2)\\
                                &-2.4619\ln(x_3) = C\end{aligned}
                                                        $} } \\  \hline
\end{tabular}
}%
    \caption{Table of values corresponding to $N=20$ for Example 3 (the left graphs of Figure \ref{fig:chemical_sing})}
    \label{tab:chemical_sing20}
\end{table}
\endgroup

\begingroup
\begin{table}[H]
    \centering
    \renewcommand{\arraystretch}{1.1}
    \footnotesize{
    \begin{tabular}{|c|c|c|c|c|c|c|}\hline
     $\varepsilon_\textbf{x}$ & Model & $\textrm{len}(\Theta)$ & $\sigma_{\textrm{cutoff}}$ & $\delta$& count & $\sum_j\lVert \Gamma*\xi_j\rVert$ \\ \hlinewd{1.5pt}
                        \multirow{12}*{0.0}
                                & (1, 0, 0) & 3.0 & 1.0e-7 &NaN & 0 & NaN\\ \cline{2-7}
                                & (2, 0, 0) & 9.0 & 1.0e-7 & 6.3238e-7 & 1.0 & 3.0538e-8\\ \cline{2-7}
                                & (3, 0, 0) & 19.0 &1.0e-7 & 4.7212e-8 & 9.0 & 6.0347e-8\\ \cline{2-7}
                                 & \textbf{(1, 0, 1)} & \textbf{6.0} &\textbf{1.0e-7 }& \textbf{0.0055963} & \textbf{1.0} &\textbf{2.1413e-8}\\ \cline{2-7}
                                & (2, 0, 1) & 12.0 & 1.0e-7 & 2.5271e-7 & 3.0 & 7.0863e-8\\ \cline{2-7}
                                & (3, 0, 1) & 22.0 & 1.0e-7  & 7.7252e-8 & 11.0 & 2.883e-8\\ \cline{2-7}
                                 & (1, 1, 0) & 9.0 & 1.0e-7  & 9.3748e-7 &1.0 & 1.9842e-8\\ \cline{2-7}
                                 & (2, 1, 0) & 15.0 & 1.0e-7  & 5.2665e-7 & 6.0 & 8.9328e-8\\ \cline{2-7}
                                & (3, 1, 0) & 25.0 & 1.0e-7  & 9.3917e-8 & 14.0 & 7.7068e-9\\ \cline{2-7}
                                & (1, 1, 1) & 12.0 & 1.0e-7 &1.2008e-6 & 3.0 & 7.9501e-8\\ \cline{2-7}
                                & (2, 1, 1) & 18.0 &  1.0e-7 &2.2626e-7 & 8.0 & 8.9132e-8\\ \cline{2-7}
                                & (3, 1, 1) & 28.0 &1.0e-7 & 7.5189e-8 & 17.0 & 5.3318e-8 \\ \hline
                                \multicolumn{7}{|c|}{\textbf{Recovered Law:} $-0.34815x_1-0.34816x_2+0.87038\ln(x_3) = C$}\\ 
                                \multicolumn{7}{|c|}{\textbf{Reduced Law:} $x_1+x_2-2.5\ln(x_3) = C$} \\ \hhline{|=|=|=|=|=|=|=|}   
                        \multirow{12}*{6.0004e-10}
                                & (1, 0, 0) & 3.0 &  0.000012322& NaN & 0 & NaN\\ \cline{2-7}
                                & (2, 0, 0) & 9.0 &  0.000021342 &  0.00018738 & 3.0 & 4.358e-6\\ \cline{2-7}
                                & (3, 0, 0) & 19.0 &0.00003101 & 0.00015757 & 12.0 & 7.9755e-6\\ \cline{2-7}
                                 & \textbf{(1, 0, 1)} & \textbf{6.0} &\textbf{0.000017426 }& \textbf{0.0055963} & \textbf{1.0} & \textbf{1.5939e-8}\\ \cline{2-7}
                                & (2, 0, 1) & 12.0 &0.000024644  & 0.00012653 & 5.0 &7.567e-6\\ \cline{2-7}
                                & (3, 0, 1) & 22.0 & 0.000033368 & 0.00012176 & 14.0 & 7.1383e-6 \\ \cline{2-7}
                                 & (1, 1, 0) & 9.0 & 0.000021342 & 0.000033311 & 2.0 & 9.7926e-7\\ \cline{2-7}
                                 & (2, 1, 0) & 15.0 & 0.000027553 & 0.000057519 & 8.0 & 6.3311e-6\\ \cline{2-7}
                                & (3, 1, 0) & 25.0 & 0.00003557  & 0.00016553 &18.0 & 0.000010562\\ \cline{2-7}
                                & (1, 1, 1) & 12.0 & 0.000024644 & 0.000024303 & 4.0 & 1.3299e-6\\ \cline{2-7}
                                & (2, 1, 1) & 18.0 & 0.000030182  & 0.0001407 & 11.0 & 0.000032848\\ \cline{2-7}
                                & (3, 1, 1) & 28.0 &0.000037644  & 0.00013138 & 20.0 & 7.676e-6  \\ \hline

                                \multicolumn{7}{|c|}{\textbf{Recovered Law:} $-0.34816x_1-0.34816x_2+0.87039\ln(x_3) = C$}\\ 
                                \multicolumn{7}{|c|}{\textbf{Reduced Law:} $x_1+x_2-2.5\ln(x_3) = C$} \\ \hhline{|=|=|=|=|=|=|=|} 
                         \multirow{12}*{6.3398e-05}
                                & (1, 0, 0) & 3.0 &0.027539& NaN & 0 & NaN\\ \cline{2-7}
                                & (2, 0, 0) & 9.0 & 0.047698  & 0.090692 & 5.0 & 0.0044301\\ \cline{2-7}
                                & (3, 0, 0) & 19.0 &0.069304  & 0.25342 & 15.0 & 0.028176\\ \cline{2-7}
                                 & \textbf{(1, 0, 1)} &\textbf{ 6.0} &\textbf{0.038946} & \textbf{0.083778} &\textbf{2.0} & \textbf{0.0055469}\\ \cline{2-7}
                                & (2, 0, 1) & 12.0 &0.055077  & 0.23125 & 8.0 & 0.020959\\ \cline{2-7}
                                & (3, 0, 1) & 22.0 &0.074575 & 0.47261 & 18.0 & 0.048114 \\ \cline{2-7}
                                 & (1, 1, 0) & 9.0 &0.047698 &0.065736 & 5.0 & 0.0038466 \\ \cline{2-7}
                                 & (2, 1, 0) & 15.0 & 0.061578  & 0.12528 & 11.0 & 0.0071169\\ \cline{2-7}
                                & (3, 1, 0) & 25.0 & 0.079497 & 0.28407 & 21.0 & 0.02934\\ \cline{2-7}
                                & (1, 1, 1) & 12.0 & 0.055077 & 0.13098 & 8.0 & 0.018194\\ \cline{2-7}
                                & (2, 1, 1) & 18.0 & 0.067456  & 0.24255& 14.0 & 0.025508\\ \cline{2-7}
                                & (3, 1, 1) & 28.0 &0.084132  & 0.47579 & 24.0 & 0.050587 \\ \hline
                            \multicolumn{7}{|c|}{\textbf{Recovered Law:} {$\!%
                                \begin{aligned} &-0.624x_1-0.277x_2+0.562x_3 +0.187\ln(x_1)\\
                                &\hspace{1em}- 0.088\ln(x_2)-0.415\ln(x_3)= C_1\\
                                &-0.379x_1-0.360x_2+0.026x_3 +0.009\ln(x_1)\\
                                &\hspace{1em}- 0.005\ln(x_2)+0.851\ln(x_3)= C_1
                                \end{aligned}
                                                        $} }\\ 
                                                        \multicolumn{7}{|c|}{}\\ 
                                \multicolumn{7}{|c|}{\textbf{Reduced Law:} {$\!%
                                \begin{aligned} &x_1-1.63x_3 -0.054\ln(x_1)\\
                                &\hspace{1em}+ 0.253\ln(x_2)+3.223\ln(x_3)= C_1\\
                                &x_2+1.642x_3 +0.542\ln(x_1)\\
                                &\hspace{1em}- 0.252\ln(x_2)-5.752\ln(x_3)= C_1
                                \end{aligned}
                                                        $} }\\\hline
\end{tabular}
}%
    \caption{Table of values corresponding to $N=100$ for Example 3 (the right graphs of Figure \ref{fig:chemical_sing})}
    \label{tab:chemical_sing100}
\end{table}
\endgroup

\subsection{Example 4}
Below are the tables corresponding to the figures in Section \ref{sec:examples-nocons}. 

\begingroup
\begin{table}[H]
    \centering
    \renewcommand{\arraystretch}{1.1}
    \footnotesize{
    \begin{tabular}{|c|c|c|c|c|c|c|}\hline
     $\varepsilon_\textbf{x}$ & Model & $\textrm{len}(\Theta)$ & $\sigma_{\textrm{cutoff}}$ & $\delta$& count & $\sum_j\lVert \Gamma*\xi_j\rVert$ \\ \hlinewd{1.5pt}
                        \multirow{12}*{0}
                                & (1, 0, 0) & 2.0 &1e-10 & NaN & 0 & NaN\\\cline{2-7}
                                & (2, 0, 0) & 5.0 &1e-10 & NaN & 0 & NaN\\ \cline{2-7}
                                & (3, 0, 0) & 9.0 &1e-10 & NaN & 0 & NaN\\ \cline{2-7}
                                & (1, 0, 1) & 4.0 &1e-10 & NaN & 0 & NaN\\ \cline{2-7}
                                & (2, 0, 1) & 7.0 &1e-10 & NaN & 0 & NaN\\\cline{2-7}
                                 & (3, 0, 1) & 11.0 &1e-10 & NaN & 0 & NaN\\\cline{2-7}
                                 & (1, 1, 0) & 6.0 &1e-10 & NaN & 0 & NaN\\\cline{2-7}
                                 & (2, 1, 0) & 9.0 &1e-10 & NaN & 0 & NaN\\ \cline{2-7}
                                & (3, 1, 0) & 13.0 &1e-10 & NaN & 0 & NaN\\\cline{2-7}
                                 & (1, 1, 1) & 8.0 &1e-10 & NaN & 0 & NaN\\ \cline{2-7}
                                & (2, 1, 1) & 11.0 &1e-10 & NaN & 0 & NaN\\ \cline{2-7}
                                & (3, 1, 1) & 15.0 &1e-10 & NaN & 0 & NaN\\ \hline
                                \multicolumn{6}{|c|}{\textbf{Recovered Law:} None}\\  
                                \hhline{|=|=|=|=|=|=|=|}
                        \multirow{12}*{4.2888e-10}
                                & (1, 0, 0) & 2.0 &3.5969e-6 & NaN & 0 & NaN\\  \cline{2-7}
                                & (2, 0, 0) & 5.0 &5.6871e-6 & NaN & 0 & NaN\\  \cline{2-7}
                                & (3, 0, 0) & 9.0 & 7.6301e-6& NaN & 0 & NaN\\  \cline{2-7}
                                & (1, 0, 1) & 4.0 &5.0867e-6  & NaN & 0 & NaN\\  \cline{2-7}
                                & (2, 0, 1) & 7.0 & 6.7291e-6& NaN & 0 & NaN\\  \cline{2-7}
                                & (3, 0, 1) & 11.0 & 8.4354e-6  & 0.000037867 & 1.0 & 4.5746e-6\\ \cline{2-7}
                                 & (1, 1, 0) & 6.0 &6.2299e-6 & NaN & 0 & NaN\\  \cline{2-7}
                                & (2, 1, 0) & 9.0 &7.6301e-6  & NaN & 0 & NaN\\  \cline{2-7}
                                & (3, 1, 0) & 13.0 &9.1702e-6 &  0.000010112 & 2.0& 1.1658e-6\\  \cline{2-7}
                                & (1, 1, 1) & 8.0 & 7.1937e-6& NaN & 0 & NaN\\  \cline{2-7}
                                & \textbf{(2, 1, 1)} &\textbf{ 11.0} & \textbf{8.4354e-6}  & \textbf{0.000089774 }& \textbf{2.0} & \textbf{0.000010243}\\  \cline{2-7}
                                & (3, 1, 1) & 15.0 &9.8504e-6  & 0.000040058 & 4.0 & 2.5907e-6\\ \hline
                                \multicolumn{7}{|c|}{\textbf{Recovered Law:} {$\!%
                                \begin{aligned} &-0.2423x_1+0.7977x_2   -0.0143x_1^2   -0.0055x_1x_2   -0.1071x_2^2   -0.0778\sin(x_1)\\   &-0.0386\sin(x_2)   -0.1794\cos(x_1)   +0.1763\cos(x_2)  +0.1742\ln(x_1)   -0.4382\ln(x_2) =C_1\\
                                & 0.7738x_1+ 0.1106x_2   -0.2916x_1^2   -0.0009x_1x_2   -0.0145x_2^2   -0.5181\sin(x_1)  \\ 
                                &-0.0040\sin(x_2)  -0.1684\cos(x_1)  + 0.0244\cos(x_2) -0.0504\ln(x_1)   -0.0624\ln(x_2)= C_2
                               \end{aligned}
                                                        $} }\\ 
                                \hhline{|=|=|=|=|=|=|=|}
                         \multirow{12}*{4.1076e-05}
                                  & (1, 0, 0) & 2.0 &0.0075294 & NaN & 0 & NaN\\   \cline{2-7}
                                & (2, 0, 0) & 5.0 &0.011905 & NaN & 0 & NaN\\   \cline{2-7}
                                & (3, 0, 0) & 9.0 & 0.015972  & 0.034424 & 3.0 & 0.0063841\\   \cline{2-7}
                                & (1, 0, 1) & 4.0 & 0.010648& NaN & 0 & NaN\\   \cline{2-7}
                                & \textbf{(2, 0, 1)} & \textbf{7.0} & \textbf{0.014086}  &\textbf{ 0.042318} & \textbf{2.0} & \textbf{0.0062048}\\   \cline{2-7}
                                & (3, 0, 1) & 11.0 & 0.017658 & 0.032374 & 5.0 & 0.013121\\   \cline{2-7}
                                & (1, 1, 0) & 6.0 & 0.013041 & 0.019994 & 1.0 & 0.0025748\\   \cline{2-7}
                                & (2, 1, 0) & 9.0 & 0.015972  & 0.045666 & 4.0 &0.013106\\   \cline{2-7}
                                & (3, 1, 0) & 13.0 & 0.019196  & 0.035502 & 7.0 & 0.0088183\\   \cline{2-7}
                                & (1, 1, 1) & 8.0 &  0.015059 & 0.02468 & 3.0 & 0.0078267\\   \cline{2-7}
                                & (2, 1, 1) & 11.0 & 0.017658 & 0.054078 & 6.0 & 0.018811\\   \cline{2-7}
                                & (3, 1, 1) & 15.0 &  0.02062  & 0.033739 & 9.0 & 0.014411\\ \hline
                                 \multicolumn{7}{|c|}{\textbf{Recovered Law:} {$\!%
                                \begin{aligned} 
                                &0.6572x_1 +   0.5337x_2   -0.1609x_1^2   -0.0370x_1x_2 \\
                                &  \hspace{1em} -0.0706x_2^2   -0.2599\ln(x_1)   -0.4283\ln(x_2) =C_1\\
                                &-0.5800x_1+    0.5550x_2  + 0.1508 x_1^2  -0.0342x_1x_2   \\
                                &\hspace{1em} -0.0735x_2^2   + 0.3492\ln(x_1)   -0.4520\ln(x_2) = C_2
                                \end{aligned}
                                                        $} }\\ \hline
\end{tabular}
}%
    \caption{Table of values corresponding to $N=20$ for Example 4 (the left graphs of Figure \ref{fig:noncons_sing})}
    \label{tab:nocons_sing20}
\end{table}
\endgroup

\begingroup
\begin{table}[H]
    \centering
    \renewcommand{\arraystretch}{1.1}
      \small{
    \begin{tabular}{|c|c|c|c|c|c|c|}\hline
     $\varepsilon_\textbf{x}$ & Model & $\textrm{len}(\Theta)$ & $\sigma_{\textrm{cutoff}}$ & $\delta$& count & $\sum_j\lVert \Gamma*\xi_j\rVert$ \\ \hlinewd{1.5pt}
                        \multirow{12}*{0}
                                & (1, 0, 0) & 2.0 &1e-10 & NaN & 0 & NaN\\\cline{2-7}
                                & (2, 0, 0) & 5.0 &1e-10 & NaN & 0 & NaN\\ \cline{2-7}
                                & (3, 0, 0) & 9.0 &1e-10 & NaN & 0 & NaN\\ \cline{2-7}
                                & (1, 0, 1) & 4.0 &1e-10 & NaN & 0 & NaN\\ \cline{2-7}
                                & (2, 0, 1) & 7.0 &1e-10 & NaN & 0 & NaN\\\cline{2-7}
                                 & (3, 0, 1) & 11.0 &1e-10 & NaN & 0 & NaN\\\cline{2-7}
                                 & (1, 1, 0) & 6.0 &1e-10 & NaN & 0 & NaN\\\cline{2-7}
                                 & (2, 1, 0) & 9.0 &1e-10 & NaN & 0 & NaN\\ \cline{2-7}
                                & (3, 1, 0) & 13.0 &1e-10 & NaN & 0 & NaN\\\cline{2-7}
                                 & (1, 1, 1) & 8.0 &1e-10 & NaN & 0 & NaN\\ \cline{2-7}
                                & (2, 1, 1) & 11.0 &1e-10 & NaN & 0 & NaN\\ \cline{2-7}
                                & (3, 1, 1) & 15.0 &1e-10 & NaN & 0 & NaN\\ \hline
                                \multicolumn{7}{|c|}{\textbf{Recovered Law:} None}\\ 
                                \hhline{|=|=|=|=|=|=|=|}
                        \multirow{12}*{3.6436e-10}
                                & (1, 0, 0) & 2.0 &7.2145e-6 & NaN & 0 & NaN\\  \cline{2-7}
                                & (2, 0, 0) & 5.0 &0.000011407 & NaN & 0 & NaN\\  \cline{2-7}
                                & (3, 0, 0) & 9.0 & 0.000015304& NaN & 0 & NaN\\  \cline{2-7}
                                & (1, 0, 1) & 4.0 &0.000010203 & NaN & 0 & NaN\\  \cline{2-7}
                                & (2, 0, 1) & 7.0 & 6.7291e-6& NaN & 0 & NaN\\  \cline{2-7}
                                & \textbf{(3, 0, 1)} & \textbf{11.0 }& \textbf{0.000016919} &\textbf{ 0.000096483} & \textbf{1.0} & \textbf{5.2896e-6}\\ \cline{2-7}
                                 & (1, 1, 0) & 6.0 & 0.000012496  & NaN & 0 & NaN\\  \cline{2-7}
                                & (2, 1, 0) & 9.0 &0.000015304  & NaN & 0 & NaN\\  \cline{2-7}
                                & (3, 1, 0) & 13.0 &0.000018393 & 0.00004666 &3.0 & 9.4759e-6\\  \cline{2-7}
                                & (1, 1, 1) & 8.0 & 0.000014429& NaN & 0 & NaN\\  \cline{2-7}
                                & (2, 1, 1) & 11.0 &0.000016919 & 0.000031836 & 2.0 & 0.000011024\\  \cline{2-7}
                                & (3, 1, 1) & 15.0 &0.000019758 & 0.000139 & 5.0 & 0.000019939\\ \hline
                                \multicolumn{7}{|c|}{\textbf{Recovered Law:}
                                {$\!%
                                \begin{aligned} &0.2608x_1   -0.8158x_2   -0.1429x_1^2  +0.0951x_1x_2  +  0.2414x_2^2+0.0281x_1^3   -0.0074x_1^2x_2\\
                                &\hspace{1em}  -0.0235x-1x_2^2 -0.0287x_2^2   -0.1261\ln(x_1) + 0.4008\ln(x_2) = C
                               \end{aligned}
                                                        $} }\\ 
                                \hhline{|=|=|=|=|=|=|=|}
                         \multirow{12}*{3.7872e-05}
                                  & (1, 0, 0) & 2.0 &0.015949 & NaN & 0 & NaN\\   \cline{2-7}
                                & \textbf{(2, 0, 0)} & \textbf{5.0} &\textbf{0.025217}  & \textbf{0.24234} & \textbf{1.0} & \textbf{0.013766}\\   \cline{2-7}
                                & (3, 0, 0) & 9.0 & 0.033832 & 0.02819 & 3.0 & 0.010148\\   \cline{2-7}
                                & (1, 0, 1) & 4.0 & 0.022555 & 0.17651 &1.0 & 0.019882\\   \cline{2-7}
                                & (2, 0, 1) & 7.0 &  0.029837  & 0.46712 & 3.0 & 0.033594\\   \cline{2-7}
                                & (3, 0, 1) & 11.0 &0.037403  & 0.019872 & 5.0 & 0.02294 \\   \cline{2-7}
                                & (1, 1, 0) & 6.0 & 0.027624 &0.23913 & 2.0 & 0.027542\\   \cline{2-7}
                                & (2, 1, 0) & 9.0 &0.033832 & 0.41062 & 5.0 & 0.048114 \\   \cline{2-7}
                                & (3, 1, 0) & 13.0 &0.040661 & 0.031443 & 7.0 & 0.014386\\   \cline{2-7}
                                & (1, 1, 1) & 8.0 &  0.031897&  0.024736 & 3.0 & 0.0092047 \\   \cline{2-7}
                                & (2, 1, 1) & 11.0 & 0.037403 & 0.5477 & 7.0 &0.062978\\   \cline{2-7}
                                & (3, 1, 1) & 15.0 & 0.043677  & 0.024525 & 9.0 & 0.024955\\ \hline
                                 \multicolumn{7}{|c|}{\textbf{Recovered Law:} $-0.5016x_1 +  0.8013x_2 +  0.2218x_1^2   -0.1194x_1x_2   -0.2070x_2^2 = C$}\\ \hline
\end{tabular}
}%
    \caption{Table of values corresponding to $N=100$ for Example 4 (the right graphs of Figure \ref{fig:noncons_sing}}
    \label{tab:nocons_sing100}
\end{table}
\endgroup

\newpage
\subsection{MAPK Pathway}
The dynamical system for the MAPK pathway as presented in \cite{Kholodenko10} is below.
\begin{equation*}
    \begin{aligned}
    &Raf\rightleftarrows pRaf & \hspace{2em} &MEK\rightleftarrows pMEK & \hspace{2em}  &ERK\rightleftarrows pERK\\    
    &pRaf \rightleftarrows ppRaf & \hspace{2em}  &pMEK\rightleftarrows ppMEK & \hspace{2em}  &pERK\rightleftarrows ppERK
    \end{aligned}
\end{equation*} 

\footnotesize{
\begin{align*}
    \dfrac{d[Raf]}{dt} &= \dfrac{V^{4}_{\textrm{max}}[pRaf]/K_{m_4}}{1+[ppRaf]/K_{m_3}+[pRaf]/K_{m_4}}-\dfrac{k_1^{\textrm{cat}}[RasGTP][Raf]/K_{m_1}}{1+[Raf]/K_{m_1}+[pRaf]/K_{m_2}}\cdot \dfrac{1+F[ppERK]/K_f}{1+[ppERK]/K_f}\\[0.5\baselineskip]
    \dfrac{d[pRaf]}{dt} &= \dfrac{k_1^{\textrm{cat}}[RasGTP][Raf]/K_{m_1}}{1+[Raf]/K_{m_1}+[pRaf]/K_{m_2}}\cdot \dfrac{1+F[ppERK]/K_f}{1+[ppERK]/K_f}
                            +\dfrac{V^{3}_{\textrm{max}}[ppRaf]/K_{m_3}}{1+[ppRaf]/K_{m_3}+[pRaf]/K_{m_4}}\\ 
                         &-\dfrac{k_2^{\textrm{cat}}[RasGTP][pRaf]/K_{m_2}}{1+[Raf]/K_{m_1}+[pRaf]/K_{m_2}}\cdot \dfrac{1+F[ppERK]/K_f}{1+[ppERK]/K_f}
                        -\dfrac{V^{4}_{\textrm{max}}[pRaf]/K_{m_4}}{1+[ppRaf]/K_{m_3}+[pRaf]/K_{m_4}}\\[0.5\baselineskip]
    \dfrac{d[ppRaf]}{dt} &= \dfrac{k_2^{\textrm{cat}}[RasGTP][pRaf]/K_{m_2}}{1+[Raf]/K_{m_1}+[pRaf]/K_{m_2}}\cdot \dfrac{1+F[ppERK]/K_f}{1+[ppERK]/K_f}
                            -\dfrac{V^{3}_{\textrm{max}}[ppRaf]/K_{m_3}}{1+[ppRaf]/K_{m_3}+[pRaf]/K_{m_4}}\\[0.5\baselineskip]
    \dfrac{d[MEK]}{dt} &= \dfrac{V^{8}_{\textrm{max}}[pMEK]/K_{m_8}}{1+[ppMEK]/K_{m_7}+[pMEK]/K_{m_8}}
                            -\dfrac{k_5^{\textrm{cat}}[ppRaf][MEK]/K_{m_5}}{1+[MEK]/K_{m_5}+[pMEK]/K_{m_6}}\\[0.5\baselineskip]
    \dfrac{d[pMEK]}{dt} &= \dfrac{k_5^{\textrm{cat}}[ppRaf][MEK]/K_{m_5}}{1+[MEK]/K_{m_5}+[pMEK]/K_{m_6}}
                            -\dfrac{k_6^{\textrm{cat}}[ppRaf][pMEK]/K_{m_6}}{1+[MEK]/K_{m_5}+[pMEK]/K_{m_6}}\\
                            &+\dfrac{V^{7}_{\textrm{max}}[ppMEK]/K_{m_7}}{1+[ppMEK]/K_{m_7}+[pMEK]/K_{m_8}}
                            -\dfrac{V^{8}_{\textrm{max}}[pMEK]/K_{m_8}}{1+[ppMEK]/K_{m_7}+[pMEK]/K_{m_8}} \\[0.5\baselineskip]
    \dfrac{d[ppMEK]}{dt} &= \dfrac{k_6^{\textrm{cat}}[ppRaf][pMEK]/K_{m_6}}{1+[MEK]/K_{m_5}+[pMEK]/K_{m_6}}
                            -\dfrac{V^{7}_{\textrm{max}}[ppMEK]/K_{m_7}}{1+[ppMEK]/K_{m_7}+[pMEK]/K_{m_8}}\\[0.5\baselineskip]
    \dfrac{d[ERK]}{dt} &= \dfrac{V_{12}^{\textrm{max}}[pERK/K_{m_{12}}]}{1+[ppERK]/K_{m_11}+[pERK]/K_{m_{12}}+[ERK]/K_{m_{13}}}
                        -\dfrac{k_9^{\textrm{cat}}[ppMEK][ERK]/K_{m_9}}{1+[ERK]/K_{m_9}+[pERK]/K_{m_{10}}}\\[0.5\baselineskip]
    \dfrac{d[pERK]}{dt} &= \dfrac{k_9^{\textrm{cat}}[ppMEK][ERK]/K_{m_9}}{1+[ERK]/K_{m_9}+[pERK]/K_{m_{10}}}
                            -\dfrac{k_{10}^{\textrm{cat}}[ppMEK][pERK]/K_{m_{10}}}{1+[ERK]/K_{m_9}+[pERK]/K_{m_{10}}}\\
                            &+\dfrac{V_{11}^{\textrm{max}}[ppERK/K_{m_{11}}]}{1+[ppERK]/K_{m_11}+[pERK]/K_{m_{12}}+[ERK]/K_{m_{13}}} 
                            -\dfrac{V_{12}^{\textrm{max}}[pERK/K_{m_{12}}]}{1+[ppERK]/K_{m_11}+[pERK]/K_{m_{12}}+[ERK]/K_{m_{13}}}\\[0.5\baselineskip]
    \dfrac{d[ppERK]}{dt} &= \dfrac{k_{10}^{\textrm{cat}}[ppMEK][pERK]/K_{m_{10}}}{1+[ERK]/K_{m_9}+[pERK]/K_{m_{10}}}
                            -\dfrac{V_{11}^{\textrm{max}}[ppERK/K_{m_{11}}]}{1+[ppERK]/K_{m_11}+[pERK]/K_{m_{12}}+[ERK]/K_{m_{13}}}
\end{align*}
}%
The following coefficients were used to generate data:
\bit
\item $V_3=2.5$, $V_4 = 3.75$, $V_7 = 3$, $V_8 = 3.75$, $V_{11}=3.75$, $V_{12}=5$ 
\item $K_{m_{1}} =100$, $K_{m_{2}}=200$, $K_{m_{3}}=50$, $K_{m_{4}}=100$, $K_{m_{5}}=250$, $K_{m_{6}}=250$, $K_{m_{7}}=80$, $K_{m_{8}}=250$, $K_{m_{9}}=250$, $K_{m_{10}}=120$, $K_{m_{11}}=20$, $K_{m_{12}}=300$
\item  $k_1=1$, $k_20.25$, $k_5=2.5$, $k_6=0.5$, $k_9=0.125$, $k_{10}=0.125$ 
\item $[RasGTP] = I = 10$
\item Negative feedback: $F= 0.17$, $K_f=15$
\item Initial values: $[Raf]_0=298$, $[pRaf]_0=1$, $[ppRaf]_0=1$, $[MEK]_0=298$, $[pMEK]_0=1$, $[ppMEK]_0=1$, $[ERK]_0=298$, $[pERK]_0=1$, $[ppERK]_0=1$
\eit

Below are the tables corresponding to the figures in Section \ref{sec:examples-mapk}.

\begingroup
\begin{table}[H]
    \centering
    \renewcommand{\arraystretch}{1.1}
    \small{
    \begin{tabular}{|c|c|c|c|c|c|c|}\hline
     $\varepsilon_\textbf{x}$ & Model & $\textrm{len}(\Theta)$ & $\sigma_{\textrm{cutoff}}$ & $\delta$& count & $\sum_j\lVert \Gamma*\xi_j\rVert$ \\ \hlinewd{1.5pt}
                        \multirow{2}*{0}
                                & \textbf{(1, 0, 0)} &\textbf{ 9.0} & \textbf{1.0e-10} & \textbf{0.039315 }& \textbf{3.0} & \textbf{7.8277e-14}\\  \cline{2-7}
                                & (1, 0, 1) & 18.0 & 1.0e-10  & 0.0008595 &3.0 & 1.1575e-13\\  \hline
                                \multicolumn{7}{|c|}{\textbf{Recovered Law:} {$\!%
                                \begin{aligned} 
                                &0.0730\sum_{i=1}^{3}x_i+0.0983\sum_{i=4}^{6}x_i+0.5642\sum_{i=7}^{9}x_i= C_1\\
                                & -0.0357\sum_{i=1}^{3}x_i-0.5668\sum_{i=4}^{6}x_i+0.1034\sum_{i=7}^{9}x_i= C_2\\
                                &0.5716\sum_{i=1}^{3}x_i-0.0480\sum_{i=4}^{6}x_i-0.0655\sum_{i=7}^{9}x_i= C_3
                                \end{aligned}
                                                        $} }\\ 
                                                        \multicolumn{7}{|c|}{}\\
                                \multicolumn{7}{|c|}{\textbf{Reduced Law:} {$\!%
                                                         \begin{aligned} &x_1+x_2+x_3 = C_1\\
                                                         &x_4+x_5+x_6=C_2\\
                                                          &x_7+x_8+x_9=C_3
                                                         \end{aligned}
                                                        $} } \\ \hhline{|=|=|=|=|=|=|=|}
                        \multirow{2}*{8.8961e-10}
                                & \textbf{(1, 0, 0)} &\textbf{ 9.0} & \textbf{0.00001241} & \textbf{0.039315 }& \textbf{3.0} &\textbf{1.0353e-10}\\ \cline{2-7}
                                & (1, 0, 1) & 18.0 & 0.00001755 & 0.0008595 & 3.0 & 2.4288e-11\\ \hline
                                \multicolumn{7}{|c|}{\textbf{Recovered Law:} {$\!%
                                \begin{aligned} &0.2457\sum_{i=1}^{3}x_i+0.5222\sum_{i=4}^{6}x_i-0.0158\sum_{i=7}^{9}x_i= C_1\\
                                & 0.5017\sum_{i=1}^{3}x_i-0.2310\sum_{i=4}^{6}x_i+0.1679\sum_{i=7}^{9}x_i= C_2\\
                                &-0.1455\sum_{i=1}^{3}x_i+0.0851\sum_{i=4}^{6}x_i+0.5521\sum_{i=7}^{9}x_i= C_3
                                \end{aligned}
                                                        $} }\\ 
                                                        \multicolumn{7}{|c|}{}\\
                                \multicolumn{7}{|c|}{\textbf{Reduced Law:} {$\!%
                                                         \begin{aligned} &x_1+x_2+x_3 = C_1\\
                                                         &x_4+x_5+x_6=C_2\\
                                                          &x_7+x_8+x_9=C_3
                                                         \end{aligned}
                                                        $} } \\ \hhline{|=|=|=|=|=|=|=|}
                         \multirow{2}*{1.0553e-04}
                                 & \textbf{(1, 0, 0)} & \textbf{9.0 }&\textbf{ 0.029962} & \textbf{0.03931} & \textbf{3.0} &\textbf{0.000011135}\\ \cline{2-7}
                                & (1, 0, 1) & 18.0 & 0.042372 & 0.016559 & 8.0 & 0.046897\\\hline
                                \multicolumn{7}{|c|}{\textbf{Recovered Law:} {$\!%
                                \begin{aligned} &0.5632\sum_{i=1}^{3}x_i+0.0379\sum_{i=4}^{6}x_i+0.1209\sum_{i=7}^{9}x_i= C_1\\
                                & -0.0237\sum_{i=1}^{3}x_i+0.5727\sum_{i=4}^{6}x_i-0.0689\sum_{i=7}^{9}x_i= C_2\\
                                &0.1244\sum_{i=1}^{3}x_i-0.0622\sum_{i=4}^{6}x_i-0.5603\sum_{i=7}^{9}x_i= C_3
                                \end{aligned}
                                                        $} }\\ 
                                                        \multicolumn{7}{|c|}{}\\
                                \multicolumn{7}{|c|}{\textbf{Reduced Law:} {$\!%
                                                         \begin{aligned} &x_1+x_2+x_3+10^{-5}\mathcal{O}(\textbf{x}) = C_1\\
                                                         &x_4+x_5+x_6+10^{-5}\mathcal{O}(\textbf{x})=C_2\\
                                                          &x_7+x_8+x_9+10^{-5}\mathcal{O}(\textbf{x})=C_3
                                                         \end{aligned}
                                                        $} } \\ \hline
\end{tabular}
}%
    \caption{Table of values corresponding to $N=20$ for MAPK example (the left graphs of Figure \ref{fig:kholodenko_sing})}
    \label{tab:kholodenko_sing20}
\end{table}
\endgroup

\begingroup
\begin{table}[H]
    \centering
    \renewcommand{\arraystretch}{1.05}
    \scriptsize{
    \begin{tabular}{|c|c|c|c|c|c|c|}\hline
     $\varepsilon_\textbf{x}$ & Model & $\textrm{len}(\Theta)$ & $\sigma_{\textrm{cutoff}}$ & $\delta$& count & $\sum_j\lVert \Gamma*\xi_j\rVert$ \\ \hlinewd{1.5pt}
                        \multirow{8}*{0.0}
                                 & \textbf{(1, 0, 0)} &\textbf{ 9.0} &\textbf{1.0e-10} & \textbf{0.40612} & \textbf{3.0 }& \textbf{5.8377e-13} \\  \cline{2-7}
                                & (2, 0, 0) & 54.0 &1.0e-10  & 9.8839e-11 & 26.0 & 1.4308e-10\\  \cline{2-7}
                                & (1, 0, 1) & 18.0 & 1.0e-10  & 0.0054736 &3.0 & 4.5144e-12\\  \cline{2-7}
                                 & (2, 0, 1) & 63.0 & 1.0e-10  & 9.7605e-11 & 26.0 & 1.3207e-10\\  \cline{2-7}   
                                 & (1, 1, 0) & 27.0 &1.0e-10& 0.20173 & 3.0 & 6.8272e-13 \\  \cline{2-7}
                                & (2, 1, 0) & 72.0 & 1.0e-10 & 8.9034e-11 & 26.0 &1.6291e-10\\  \cline{2-7}
                                & (1, 1, 1) & 36.0 & 1.0e-10 & 0.0011957 & 3.0 & 6.3639e-12\\  \cline{2-7}
                                & (2, 1, 1) & 81.0 & 1.0e-10 &8.6342e-11 & 26.0 & 1.5094e-10 \\ \hline
                               \multicolumn{7}{|c|}{\textbf{Recovered Law:} {$\!%
                                \begin{aligned} &-0.0120\sum_{i=1}^{3}x_i+0.0010\sum_{i=4}^{6}x_i+0.5772\sum_{i=7}^{9}x_i= C_1\\
                                & 0.1357\sum_{i=1}^{3}x_i-0.5611\sum_{i=4}^{6}x_i+0.0038\sum_{i=7}^{9}x_i= C_2\\
                                &0.5610\sum_{i=1}^{3}x_i+0.1357\sum_{i=4}^{6}x_i+0.0114\sum_{i=7}^{9}x_i= C_3
                                \end{aligned}
                                                        $} }\\ 
                                                        \multicolumn{7}{|c|}{}\\
                                \multicolumn{7}{|c|}{\textbf{Reduced Law:} {$\!%
                                                         \begin{aligned} &x_1+x_2+x_3 = C_1\\
                                                         &x_4+x_5+x_6=C_2\\
                                                          &x_7+x_8+x_9=C_3
                                                         \end{aligned}
                                                        $} } \\  \hhline{|=|=|=|=|=|=|=|}   
                        \multirow{8}*{1.4759e-09}
                                 & \textbf{(1, 0, 0) }& \textbf{9.0} & \textbf{0.000038889 } &\textbf{ 0.40612 }& \textbf{3.0} &\textbf{2.2547e-9}\\  \cline{2-7}
                                & (2, 0, 0) & 54.0 &0.000095258  & 0.11421 & 27.0 & 8.5191e-7\\  \cline{2-7}
                                & (1, 0, 1) & 18.0 &0.000054997  &0.0054736 & 3.0 & 2.0546e-9 \\  \cline{2-7}
                                 & (2, 0, 1) & 63.0 & 0.00010289  & 0.000073636 & 30.0 &0.00005838\\  \cline{2-7}   
                                 & (1, 1, 0) & 27.0 & 0.000067358  & 0.20173 & 3.0 & 2.0246e-9\\  \cline{2-7}
                                & (2, 1, 0) & 72.0 & 0.00010999  & 0.042939 & 27.0 & 7.0647e-7\\  \cline{2-7}
                                & (1, 1, 1) & 36.0 &0.000077778 & 0.0011957 & 3.0 & 1.5459e-9\\  \cline{2-7}
                                & (2, 1, 1) & 81.0 & 0.00011667 & 0.00011868 & 32.0 & 0.0002048 \\ \hline
                                \multicolumn{7}{|c|}{\textbf{Recovered Law:} {$\!%
                                \begin{aligned} &0.4287\sum_{i=1}^{3}x_i+0.3669\sum_{i=4}^{6}x_i-0.1216\sum_{i=7}^{9}x_i= C_1\\
                                & 0.3191\sum_{i=1}^{3}x_i-0.4385\sum_{i=4}^{6}x_i-0.1975\sum_{i=7}^{9}x_i= C_2\\
                                &0.2180\sum_{i=1}^{3}x_i-0.0794\sum_{i=4}^{6}x_i+0.5286\sum_{i=7}^{9}x_i= C_3
                                \end{aligned}
                                                        $} }\\ 
                                                        \multicolumn{7}{|c|}{}\\
                                \multicolumn{7}{|c|}{\textbf{Reduced Law:} {$\!%
                                                         \begin{aligned} &x_1+x_2+x_3 = C_1\\
                                                         &x_4+x_5+x_6=C_2\\
                                                          &x_7+x_8+x_9=C_3
                                                         \end{aligned}
                                                        $} } \\  \hhline{|=|=|=|=|=|=|=|} 
                         \multirow{8}*{1.2644e-04}
                                 & \textbf{(1, 0, 0)} & \textbf{9.0} & \textbf{0.075573}  & \textbf{0.40605} & \textbf{3.0 }&\textbf{0.00019149}\\  \cline{2-7}
                                & (2, 0, 0) & 54.0 &0.18512  & 1.4879 & 29.0 & 0.37502\\  \cline{2-7}
                                & (1, 0, 1) & 18.0 & .10688 & 0.32461 &11.0 & 0.26634\\  \cline{2-7}
                                 & (2, 0, 1) & 63.0 & 0.19995& 0.084683 & 37.0 & 0.30556\\  \cline{2-7}   
                                 & (1, 1, 0) & 27.0 & 0.1309  & 0.20166& 3.0 & 0.00017206\\  \cline{2-7}
                                & (2, 1, 0) & 72.0 &0.21375 & 0.13573 & 31.0 & 0.4558 \\  \cline{2-7}
                                & (1, 1, 1) & 36.0 &0.15115 &0.1464 & 11.0 & 0.13342\\  \cline{2-7}
                                & (2, 1, 1) & 81.0 & 0.22672& 0.13743 & 40.0 & 0.50621 \\ \hline
                                \multicolumn{7}{|c|}{\textbf{Recovered Law:} {$\!%
                                \begin{aligned} &-0.1896\sum_{i=1}^{3}x_i+0.5162\sum_{i=4}^{6}x_i+0.1756\sum_{i=7}^{9}x_i= C_1\\
                                & -0.5134\sum_{i=1}^{3}x_i-0.2316\sum_{i=4}^{6}x_i+0.1265\sum_{i=7}^{9}x_i= C_2\\
                                &0.1836\sum_{i=1}^{3}x_i-0.1146\sum_{i=4}^{6}x_i+0.5352\sum_{i=7}^{9}x_i= C_3
                                \end{aligned}
                                                        $} }\\
                                                        \multicolumn{7}{|c|}{}\\
                                \multicolumn{7}{|c|}{\textbf{Reduced Law:} {$\!%
                                                         \begin{aligned} &x_1+x_2+x_3 +10^{-5}\mathcal{O}(\textbf{x})= C_1\\
                                                         &x_4+x_5+x_6+10^{-5}\mathcal{O}(\textbf{x})=C_2\\
                                                          &x_7+x_8+x_9+10^{-5}\mathcal{O}(\textbf{x})=C_3
                                                         \end{aligned}
                                                        $} } \\ \hline
\end{tabular}
}%
    \caption{Table of values corresponding to $N=100$ for MAPK example (the right graphs of Figure \ref{fig:kholodenko_sing})}
    \label{tab:kholodenko_sing100}
\end{table}
\endgroup

\end{document}